\documentclass[8pt]{article}
\pdfoutput=1 
\usepackage[english]{babel}
\usepackage{amsmath}
\usepackage{amssymb}
\usepackage{theorem}
\usepackage{soul}

\usepackage{afterpage}
\usepackage[retainorgcmds]{IEEEtrantools}
\usepackage{graphicx}
\usepackage{subfigure}
\graphicspath{{figss/}}
\usepackage{color}
\usepackage{pstricks,pst-all,pstricks-add,pst-plot,pst-eucl}%,pst-math}
\usepackage{enumerate}

\theorembodyfont{\it}
\newtheorem{theorem}{Theorem}

\newtheorem{lemma}{Lemma}

\theorembodyfont{\rm}

\newtheorem{remark}{Remark}

\makeatletter
\renewcommand{\fnum@figure}{\small\textbf{\figurename~\thefigure}}
\makeatother

\definecolor{brown}{rgb}{0.64,0.16,0.16}
\definecolor{ForestGreen}{rgb}{0.13,0.54,0.13}
\definecolor{purple}{rgb}{0.62,0.12,0.94}
\definecolor{DodgerBlue}{rgb}{0.11,0.56,0.98}
\definecolor{RoyalBlue}{rgb}{0.25,0.41,0.88}
\definecolor{B}{rgb}{0,0,1}
\definecolor{G}{rgb}{0,0.502,0}
\definecolor{R}{rgb}{1,0,0}

\def \RR {\mathbb{R}}

\title{An organizing center in a planar model of\\neuronal excitability}
\author{Alessio Franci$^{1,\dag}$, Guillaume Drion$^{2,1,\dag}$, \& Rodolphe Sepulchre$^{1}$\\
\small{$^1$Department of Electrical Engineering and Computer Science, University of Li\`ege, Li\`ege, Belgium.}\\
\small{$^2$Laboratory of Pharmacology and GIGA Neurosciences, University of Li\`ege, Li\`ege, Belgium.}\\
\small{$^\dag$These authors contributed equally to this work}}
\date{}

\begin{document}

\maketitle

\begin{abstract}
The paper studies the excitability properties of a generalized FitzHugh-Nagumo model. The model differs from the purely competitive FitzHugh-Nagumo model in that it accounts for the effect of cooperative gating variables such as activation of calcium currents. Excitability is explored by unfolding a pitchfork bifurcation that is shown to organize five different types of excitability. In addition to the three classical types of neuronal excitability, two novel types are described and distinctly associated to the presence of cooperative variables.
\end{abstract}

\section{Introduction}

At the root of neuronal signaling, excitability is a dynamical property shared by all neurons, but its electrophysiological signature largely differs across neurons and experimental conditions. In the early days of experimental neurophysiology, Hodgkin \cite{hodgkin1948local} identified three distinct types of excitability (nowadays called Type I, II, and III) by stimulating crustacean nerves with constant current stimuli. The three types of excitability have long been associated to three distinct mathematical signatures in conductance-based models. They all can be described in planar models of the FitzHugh \cite{fitzhugh61} type, which can be rigorously associated to the mathematical reduction of high-dimensional models (see for instance the planar reduction of Hodgkin-Huxley model in \cite{rinzel1985excitation} and the excitability analysis in \cite{Rinzel:1989:ANE:94605.94613}). Reduced models have proven central to the understanding of excitability and closely related mechanisms such as bursting. However, understanding excitability in detailed conductance-based models remains a challenge, especially for neurons that exhibit transition between distinctively different firing types depending on environmental conditions. The present paper proposes a generalization of FitzHugh-Nagumo model that provides novel insights in the simple classification of excitability types.

The proposed model is a ``mirrored'' version of FitzHugh-Nagumo model, motivated by the mathematical reduction of conductance-based models including calcium channels. A central observation in \cite{DRFRSESE_PLoS_sub} is that the cooperative nature of calcium channel activation makes their contribution to excitability fundamentally different from competitive gating variables, such as sodium inactivation or potassium activation. This chief difference is responsible for an alteration of the phase portrait that cannot be reproduced in FitzHugh-Nagumo model and that calls for a generalized model that motivates the present paper.

We construct a highly degenerate pitchfork bifurcation (co-dimension 3) that is shown to organize excitability in five different types. The three first types correspond to the types of excitability extensively studied in the literature. They are all competitive, in the sense that they only involve a region of the phase plane where our model is purely competitive. In addition, the model reveals two new types of excitability (Type IV and V) that match the distinct electrophysiological signatures of conductance-based models of high density calcium channels. We prove that these two new types of excitability cannot be observed in a purely competitive model, such as FitzHugh-Nagumo model.

Global phase portraits of the proposed model are studied using singular perturbations theory, exploiting the timescale separation observed in all physiological recordings. An important result of the analysis is that Type IV and V excitable models exhibit a bistable range that persists in the singular limit of the model. This, in sharp contrast to the Type I, II, and III. The result suggests the potential importance of Type IV and V excitability in bursting mechanisms associated to cooperative ion channels.

The paper is organized as follows. The studied planar model, its geometrical properties, and the underlying physiological concepts are presented in Section \ref{SEC: proposed model and planar reduction}. In Section \ref{SEC: cod 3 bif}, we construct the pitchfork bifurcation organizing the model and unfold it into the five types of excitability. Types I, II, and III are briefly reviewed in Section \ref{SEC: known excitability types}. Types IV and V are defined and characterized via phase plane and numerical bifurcation analysis in Section \ref{SEC: novel excitability types}. By relying on geometrical singular perturbations, we provide in Section \ref{SEC: Type IV and V phase portrait} a qualitative description of the phase plane of Type IV and V excitable models. Their peculiarities with respect to Type I, II, and III are particularly stressed. A summary and a discussion of our analysis are provided in Section \ref{SEC: discussion}.

\section{A mirrored FitzHugh-Nagumo model and its physiological interpretation}
\label{SEC: proposed model and planar reduction}

The paper studies the excitability properties of the planar model
\begin{IEEEeqnarray}{rCl}\label{EQ: mirrored FHN dynamics}
\dot V&=&V-\frac{V^3}{3}-n^2+I_{app}\IEEEyessubnumber\\
\dot n&=&\epsilon(n_{\infty}(V-V_0)+n_0-n)\IEEEyessubnumber
\end{IEEEeqnarray}
where $n_\infty(V)$ is the standard Boltzman activation function
\begin{equation}\label{EQ: sigmoid function}
%n_{\infty}(V):=\frac{a}{1+e^{-4k_0V/a}},
n_{\infty}(V):=\frac{2}{1+e^{-5V}},
\end{equation}
The model is reminiscent of the popular FitzHugh-Nagumo model of neuronal excitability: equation (\ref{EQ: mirrored FHN dynamics}a) describes the fast dynamics of the membrane potential $V$, whereas (\ref{EQ: mirrored FHN dynamics}b) describes the slow dynamics of the ``recovery variable'' $n$ that aggregates the gating of various ionic channels.

The voltage dynamics (\ref{EQ: mirrored FHN dynamics}a) are identical to the FitzHugh-Nagumo model, except that the quadratic term $n^2$ replaces the linear term $n$. This sole modification is central to the result of the present paper. The resulting nullcline $\dot V=0$ ``mirrors'' along the $V$-axis the classical inverse $N$-shaped nullcline of FitzHugh-Nagumo model, as illustrated in Figure \ref{FIG: trans organizing center}. The figure illustrates the phase portrait of the model in the singular limit $\epsilon=0$. The left and right phase-portraits are the unfolding of the transcritical bifurcation organizing center (see e.g. \cite[Pages 104-105]{SEYDEL94}) obtained for $I_{app}=I^\star:=\frac{2}{3}$. This particular value will help understanding excitability mechanisms at work in the situation $I_{app}> I^\star$, illustrated in the right figure.

\begin{figure}[h!]
\centering
%\subfigure[][{\large $I_{app}<I_{tc}=\frac{2}{3}$. No jumps.}]{
%\includegraphics[width=0.32\textwidth]{V_null1}
%}
%\subfigure[][{\large $I_{app}=I_{tc}$. Transcritical organizing center}]{
%\includegraphics[width=0.32\textwidth]{V_null2}
%}
%\subfigure[][{\large $I_{app}<I_{tc}=\frac{2}{3}$. 2 fold jumps.}]{
%\includegraphics[width=0.32\textwidth]{V_null3_no_frame}
%}
\includegraphics[width=\textwidth]{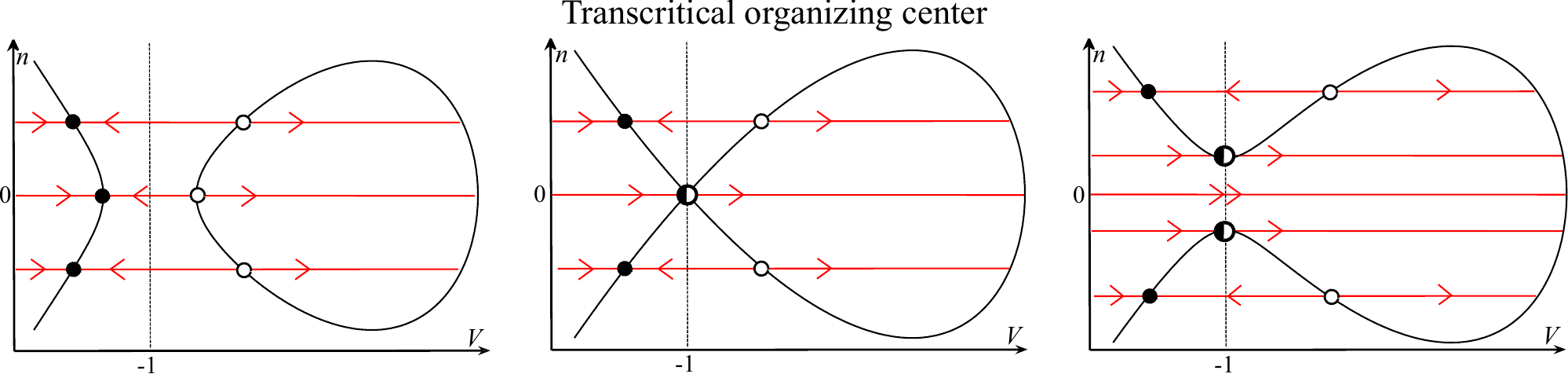}\\
\noindent{$I_{app}<I^\star$\hspace{0.25\textwidth}$I_{app}=I^\star$\hspace{0.25\textwidth}$I_{app}>I^\star$\hspace{0.3\textwidth}}
\vspace{-6mm}
\caption{$V$-nullclines of (\ref{EQ: mirrored FHN dynamics}) for different values of $I_{app}$ and $\epsilon=0$. Stable fixed points are depicted as filled circles, unstable as circles, and bifurcations as half filled circles}\label{FIG: trans organizing center}
\end{figure}

The recovery dynamics (\ref{EQ: mirrored FHN dynamics}b) exhibit the familiar firts-order relaxation of ionic current to the static sigmoid curve illustrated in Figure \ref{FIG: n relax curve and inv region}. For quantitative purposes, the numerator in the right hand side of (\ref{EQ: sigmoid function}) can be picked larger or equal than 2 without changing the underlying qualitative analysis. The parameters $(V_0,n_0)$ locate the relative position of the nullclines in the phase portrait. In particular, the region
\begin{equation}\label{EQ: invariant physiological strip}
S_{n_0}:=\{(V,n)\in\RR^2:\ n\in(n_0,\ n_0+2)\}
\end{equation}
is attractive and invariant for the dynamics (\ref{EQ: mirrored FHN dynamics}). The parameter $n_0$ slides up and down the ``physiological window'' of the recovery variable, whereas, $V_0$ is the half-activation voltage. The voltage $V^\star<V_0$ is defined as the voltage at which $n_{\infty}(V-V_0)$ has unitary slope. We adopt the conventional notation $n$ for the recovery variable but will allow $n_0<0$, which makes the range of $n$ include negative values. This is purely for mathematical convenience and should not confuse the reader used to the physiological interpretation of $n$ as a gating variable with range $[0,1]$.
\begin{figure}[h!]
\centering
\includegraphics[width=0.66\textwidth]{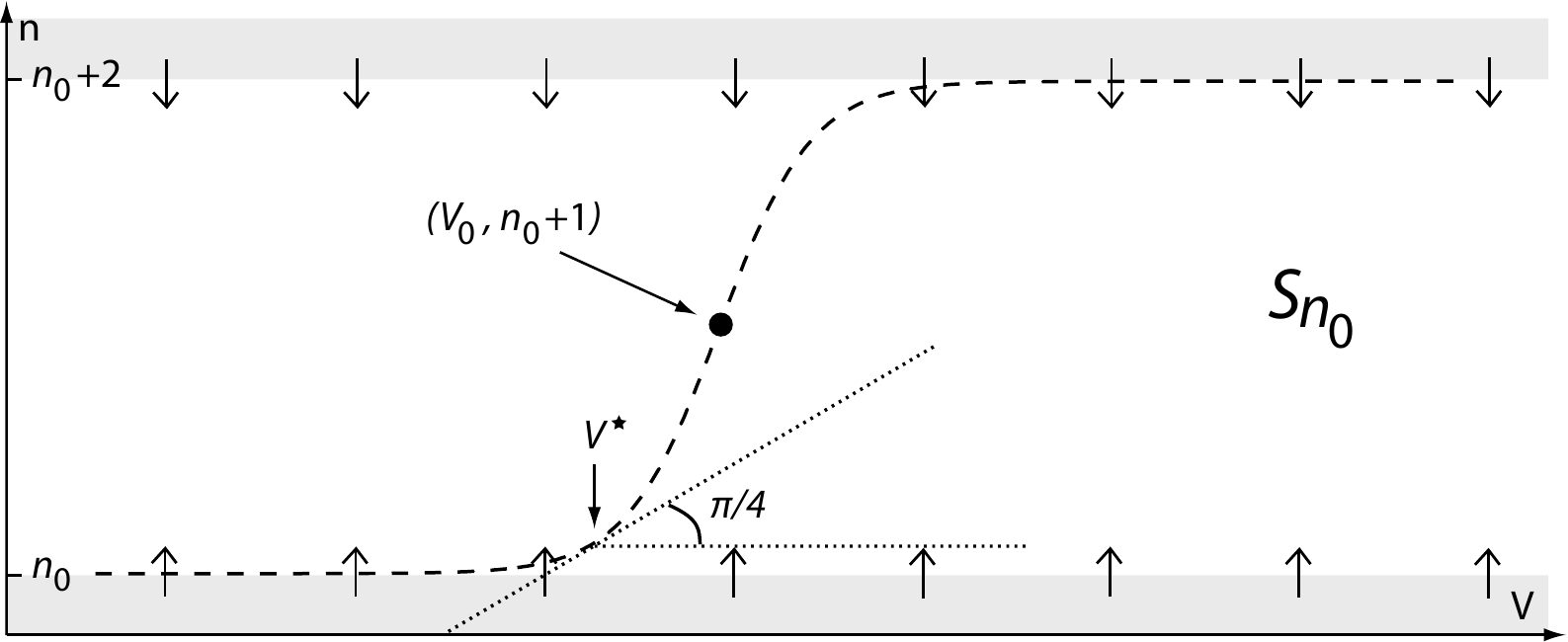}
\caption{Dependence of the $n$-nullclines of (\ref{EQ: mirrored FHN dynamics}) on the parameters $n_0$ and $V_0$. The location of the nullcline in the phase plane determines an attractive invariant region $S_{n_0}$. At the voltage $V^\star$, the nullcline has unitary slope.}\label{FIG: n relax curve and inv region}
\end{figure}

%In the Conclusion and Perspective section \af{??} we propose a 6-parameters qualitatively equivalent extension of (\ref{EQ: mirrored FHN dynamics}), which permits to better quantitatively fit neuronal spiking.

%\begin{remark}\label{RMK: unimportant parameters}
%%{\bf: the parameters $a,\,k_0$.}
%%In the remainder of the paper, we let $a$ and $k_0$ be fixed, and consider $I_{app},\,V_0,\,n_0$ as the only free parameters. Qualitatively, the analysis is independent of $a$ and $k_0$, provided that they are sufficiently large. More precisely, we fix $a\geq 2$ and $k>1$. The reason for these choices will be clearer in the sequel.
%{\bf: A more quantitative model.} The model (\ref{EQ: mirrored FHN dynamics}) is a qualitative model depending only
%\end{remark}
%When $k_0>1$, we can well define the voltage $V^\star<V_0$ at which the slope of the activation function $n_\infty(V-V_0)$ is unitary, {\it i.e.} $\left.\frac{\partial n_\infty(V-V_0)}{\partial V}\right|_{V=V^\star}=1$, as sketched in Figure \ref{FIG: n relax curve and inv region}.

\subsection*{Competitive and cooperative excitability}

In the terminology of \cite{smith2008monotone}, the variables $V$ and $n$ of model (\ref{EQ: mirrored FHN dynamics}) are competitive in the half plane $n>0$, that is the model satisfies
\begin{equation*}
\frac{\partial\dot V}{\partial n}\frac{\partial\dot n}{\partial V}<0,
\end{equation*}
whereas they are cooperative in the half plane $n<0$, that is
\begin{equation*}
\frac{\partial\dot V}{\partial n}\frac{\partial\dot n}{\partial V}>0.
\end{equation*}
A consequence of that observation is that model (\ref{EQ: mirrored FHN dynamics}) is purely competitive when $n_0>0$. The corresponding phase portrait is illustrated in Figure \ref{FIG: n0 role and physio region}A. It is reminiscent of FitzHugh-Nagumo model and will recover known types of excitability. In contrast, the model is neither competitive nor cooperative when $n_0<0$. The corresponding phase portrait (Figure \ref{FIG: n0 role and physio region}B) is distinctly different from FitzHugh-Nagumo model and it will lead to the novel types of excitability studied in this paper. The role of the proposed mirrored FitzHugh-Nagumo model is to study the transition from a purely competitive model to a model that is neither cooperative nor competitive through a single parameter $n_0$. 

\begin{figure}
\centering
\includegraphics[width=0.9\textwidth]{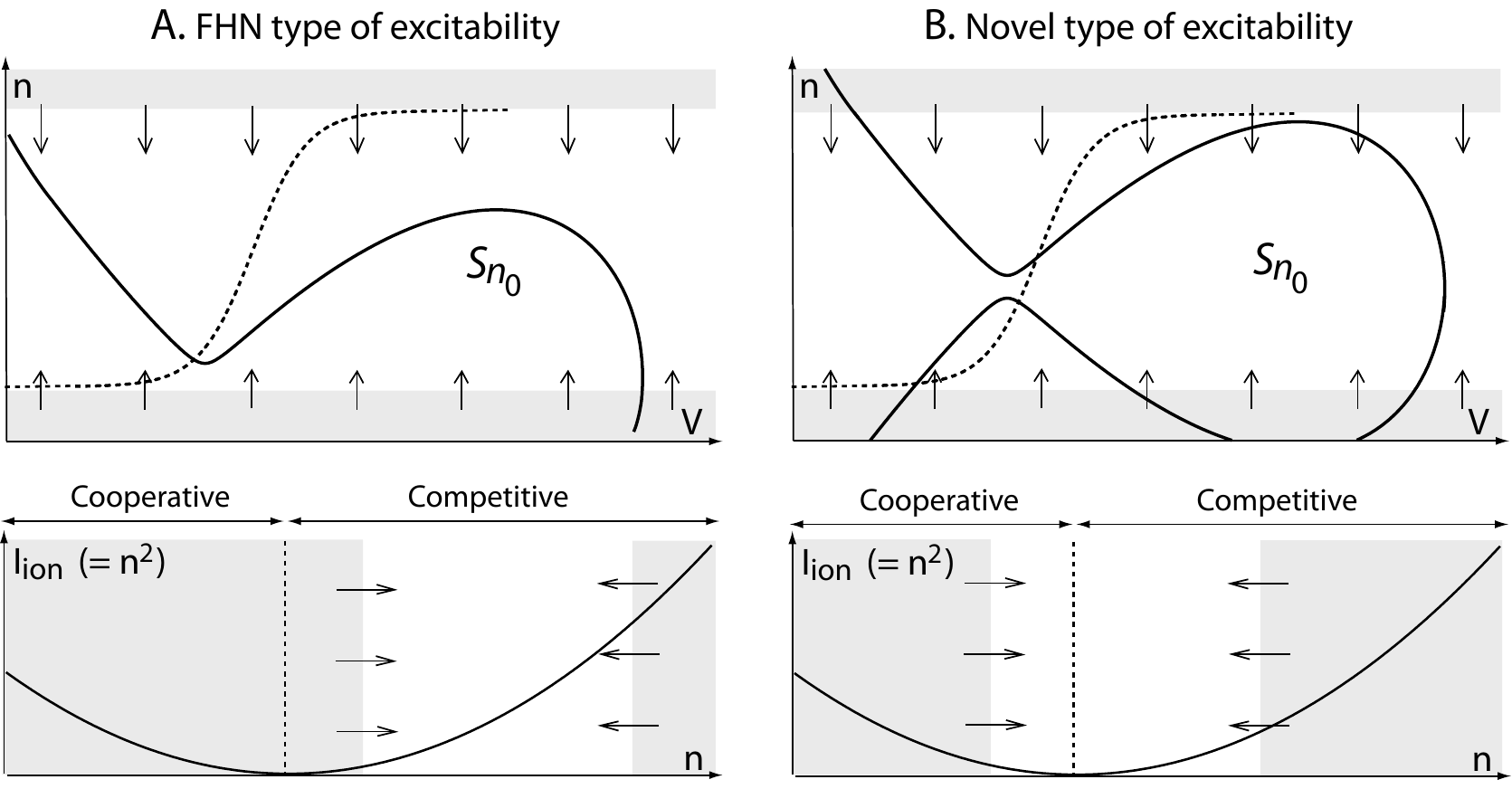}
\caption{Phase portrait and outward ionic current $n^2$ for different positions of the $n$-nullcline and the associated invariant region $S_{n_0}$. A) When $S_{n_0}$ is fully contained in the half-plane $n>0$ the phase portrait is reminiscent of original FitzHugh-Nagumo model and the outward ionic current is monotone increasing. The recovery variable has a purely competitive role. B) When $S_{n_0}$ extends to the half-plane $n<0$, the phase portrait exhibits new characteristics. The outward ionic current is not monotone, corresponding to two antisynergistic (competitive and cooperative) roles of the recovery variable.}\label{FIG: n0 role and physio region}
\end{figure}

\subsection*{The physiology behind competitive and cooperative behaviors}

Competitive and cooperative behaviors model different types of (in)activation gating variables in conductance-based models. When the recovery variable is purely competitive (Figure \ref{FIG: n0 role and physio region}A), it models the activation (resp. inactivation) of an outward (resp. inward) ionic current: a positive variation of $V$ induces a positive variation of $n$ and thus an \emph{increase} of the total outward ionic current, {\it i.e.} $n^2$, which is monotone increasing in $S_{n_0}$. In contrast, when the recovery variable becomes cooperative (Figure \ref{FIG: n0 role and physio region}B), its role is reversed: it models the activation (resp. inactivation) of an inward (resp. outward) ionic current: a positive variation of $V$ induces a positive variation of $n$ and thus a \emph{decrease} of the total outward ionic current, which is now monotone decreasing.

The seminal model of Hodgkin-Huxley only includes competitive slow gating variables: inactivation of sodium current and activation of potassium current. That is why the classical reduction of the Hodgkin-Huxley model leads to a FitzHugh-Nagumo type of phase portrait (Figure \ref{FIG: n0 role and physio region}A). However, conductance-based models often include cooperative gating variables. An example of the latter is the activation of calcium currents included in many bursting conductance-based models (e.g. R15 neuron of Aplysia's abdominal ganglion \cite{plant1976mathematical}, thalamo-cortical relay and reticular neurons \cite{McCormick92a,destexhe1996vivo}, CA3 hippocampal pyramidal neuron \cite{Traub01081991}). Adding a calcium current in the Hodgkin-Huxley model is one natural way to obtain a reduced phase portrait as in Figure \ref{FIG: n0 role and physio region}B. This observation, firstly presented in \cite{DRFRSESE_PLoS_sub} and reproduced in Figure \ref{FIG: trans in HH}, motivated the present study.

\begin{figure}
\centering
\includegraphics[width=0.9\textwidth]{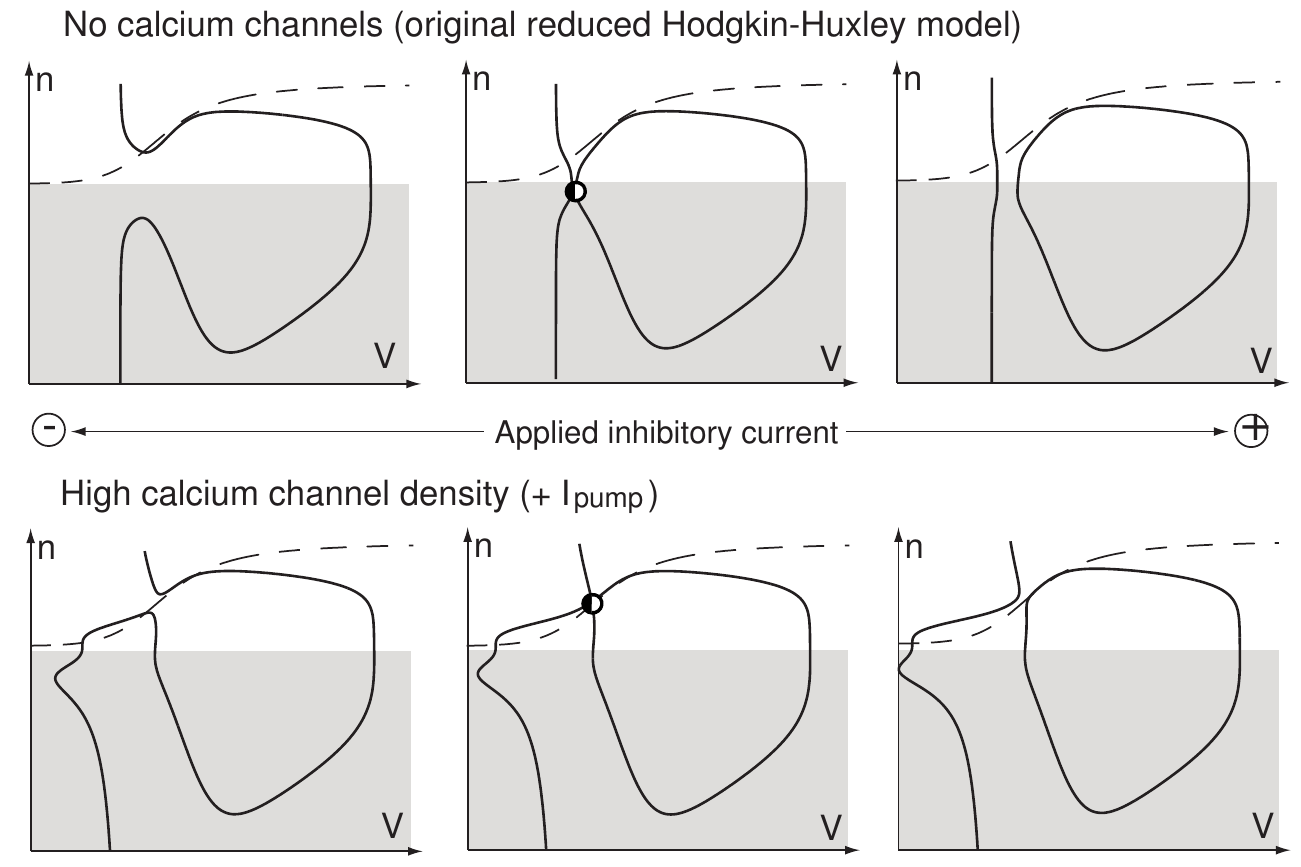}
\caption{In original reduced Hodgkin-Huxley model a transcritical singularity exists in the non-physiological region of the phase plane. The addition of a calcium current makes this bifurcation physiological.}\label{FIG: trans in HH}
\end{figure}

\section{A pitchfork bifurcation organizes different excitability types}
\label{SEC: cod 3 bif}

The model (\ref{EQ: mirrored FHN dynamics}) has three free geometrical parameters ($I_{app},n_0,V_{0}$). The parameter $n_0$ is an additional parameter with respect to FitzHugh-Nagumo dynamics. The three parameters can be adjusted to create a codimension three bifurcation that will provide an organizing center for excitability.

The degenerate bifurcation is illustrated in Figure \ref{FIG: bif geom constr} and is constructed as follows:
\begin{enumerate}
\item The applied current is fixed at $I_{app}=I^\star$, imposing the existence of the singularly perturbed transcritical bifurcation at the $V$-nullcline self-intersection (cf. Figure \ref{FIG: trans organizing center} center).
\item We fix $n_0=n_0^\star(V_0):=-n_\infty(-1-V_0)$ to force a nullcline intersection at the transcritical singularity at $(-1,0)$.
\item We fix $V^\star=-1$, or, equivalently, $V_0=V_0^\star:=-1+\frac{1}{5}\log\left(-6+\sqrt{35}\right)$, so that the $n$-nullcline is tangent to the $V$-nullcline at the intersection.
\begin{figure}[h]
\center
\includegraphics[width=0.4\textwidth]{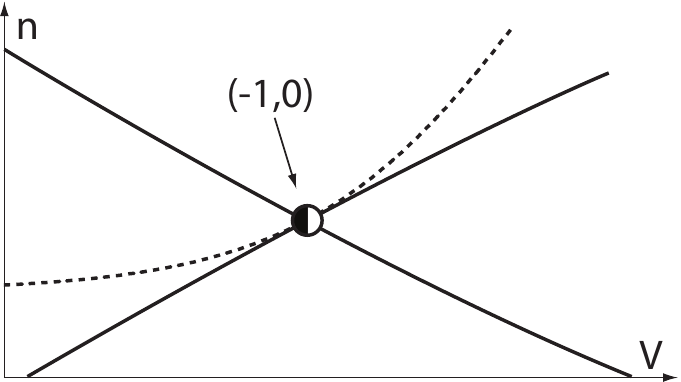}
\caption{Phase portrait of (\ref{EQ: mirrored FHN dynamics}) for $I_{app}=I^\star$, $n_0=n_0^\star(V_0)$, and $V^\star=-1$.}\label{FIG: bif geom constr}
\end{figure}
\end{enumerate}

A local normal form of the model (\ref{EQ: mirrored FHN dynamics}) at this degenerate bifurcation is provided by the following lemma.
\begin{lemma}
There exists an affine change of coordinates that transforms (\ref{EQ: mirrored FHN dynamics}) into
\begin{IEEEeqnarray}{rCl}\label{EQ: pitch normal form full}
\dot v&=&\frac{v^2(3-v)}{3}-\Big(k(V_0)(v-u)+\delta_0\Big)^2+I_{app}-I^\star\IEEEyessubnumber\\
\dot u&=&-\epsilon u+\mathcal O(v^2,vu,u^2)),\IEEEyessubnumber
\end{IEEEeqnarray}
where $k(V_0):=\frac{d n_{\infty}}{dV}(-1-V_0)$ is the slope of the activation function $n_{\infty}(V-V_0)$ at $V=-1$, and $\delta_0=n_0+n_\infty(-1-V_0)$.
\end{lemma}
\begin{proof}
The affine change of coordinates
\begin{IEEEeqnarray*}{rCl}
\tilde w&:=&n-\delta_0,\\
v&=&V+1,
\end{IEEEeqnarray*}
transforms (\ref{EQ: mirrored FHN dynamics}) into
\begin{IEEEeqnarray}{rCl}\label{EQ: pitch norm form proof 1}
\dot v&=&\frac{v^2(3-v)}{3}-(\tilde w+\delta_0)^2+I_{app}-I^\star\nonumber\\
\dot{\tilde w}&=&\epsilon(n_{\infty}(v-1-V_0)-n_{\infty}(-1-V_0)-\tilde w),
\end{IEEEeqnarray}
where $\delta_0$ is defined as in the statement of the lemma. Because $\dot{\tilde w}|_{v=\tilde w=0}=0$, we extract the linear term in (\ref{EQ: pitch norm form proof 1}) and write
\begin{IEEEeqnarray}{rCl}\label{EQ: pitchfork main taylor}
\dot v&=&\frac{v^2(3-v)}{3}-(\tilde w+\delta_0)^2+I_{app}-I^\star\IEEEyessubnumber\\
\dot{\tilde w}&=&\epsilon(k(V_0)v-\tilde w+\mathcal O(v^2,v\tilde w,\tilde w^2)),\IEEEyessubnumber
\end{IEEEeqnarray}
Finally, the linear transformation $u=v-\frac{\tilde w}{k(\bar V_0)}$ transforms (\ref{EQ: pitchfork main taylor}) into
\begin{IEEEeqnarray*}{rCl}
\dot v&=&\frac{v^2(3-v)}{3}-(k(V_0)(v-u)+\delta_0)^2+I_{app}-I^\star\\
\dot u&=&-\epsilon u+\mathcal O(v^2,vu,u^2)),
\end{IEEEeqnarray*}
which proves the lemma.
\end{proof}

The dynamics (\ref{EQ: pitch normal form full}) have an exponentially attractive center manifold that is tangent at $v=u=0$ to center space $\{(v,u):\ u=0\}$. Ignoring higher order terms, the dynamics on the center manifold is given by
\begin{equation}\label{EQ: pitch normal form center}
\dot v=\big(1-k(V_0)^2\big)v^2-\frac{v^3}{3}-2k(V_0)v\delta_0-\delta_0^2+I_{app}-I^\star.
\end{equation}

\subsubsection*{Codimension 2 transcritical bifurcation for $I_{app}=I^\star$, $n_0=n_0^\star(V_0)$, $V_0\neq V_0^\star$}

Fixing $I_{app}=I^\star$ (Item 1.) and $\delta_0=0$ (Item 2.), but $V_0\neq V_0^\star$, we obtain from (\ref{EQ: pitch normal form center}) the following relationships:
\begin{IEEEeqnarray}{lll}\label{EQ: trans def conds}
\dot v\left|_{\substack{v=\delta=0\\
I_{app}=I^\star\\
V_0\neq V_0^\star}}\right.=\frac{\partial\dot v}{\partial v}\left|_{\substack{v=\delta=0\\
I_{app}=I^\star\\
V_0\neq V_0^\star}}\right.=\frac{\partial\dot v}{\partial \delta_0}\left|_{\substack{v=\delta=0\\
I_{app}=I^\star\\
V_0\neq V_0^\star}}\right.=0,\IEEEyessubnumber\\
\frac{\partial^2\dot v}{\partial v^2}\left|_{\substack{v=\delta=0\\
I_{app}=I^\star\\
V_0\neq V_0^\star}}\right.=2\big(1-k(V_0)^2\big)\neq0,\IEEEyessubnumber\\
\frac{\partial^2\dot v}{\partial v\partial\delta_0}\left|_{\substack{v=\delta=0\\
I_{app}=I^\star\\
V_0\neq V_0^\star}}\right.=2k(V_0)\neq0,\IEEEyessubnumber
\end{IEEEeqnarray}
where (\ref{EQ: trans def conds}b) comes from the fact that, since $V_0\neq V_0^\star$, $k(V_0)\neq 1$, whereas (\ref{EQ: trans def conds}c) comes from the fact that $k(V_0)\neq0$ for all $V_0\in\RR$. Relation (\ref{EQ: trans def conds}) are the defining conditions of a transcritical singularity (see e.g. \cite[Page 367]{SEYDEL94}). For any value of the perturbation parameter $\delta_0\neq0$, there are two fixed points that exchange their stability at the bifurcation for $\delta_0=0$. Figure \ref{FIG: trans and pitch} {\bf A,C} illustrates this result.
\begin{figure}[h]
\center
\includegraphics[width=0.8\textwidth]{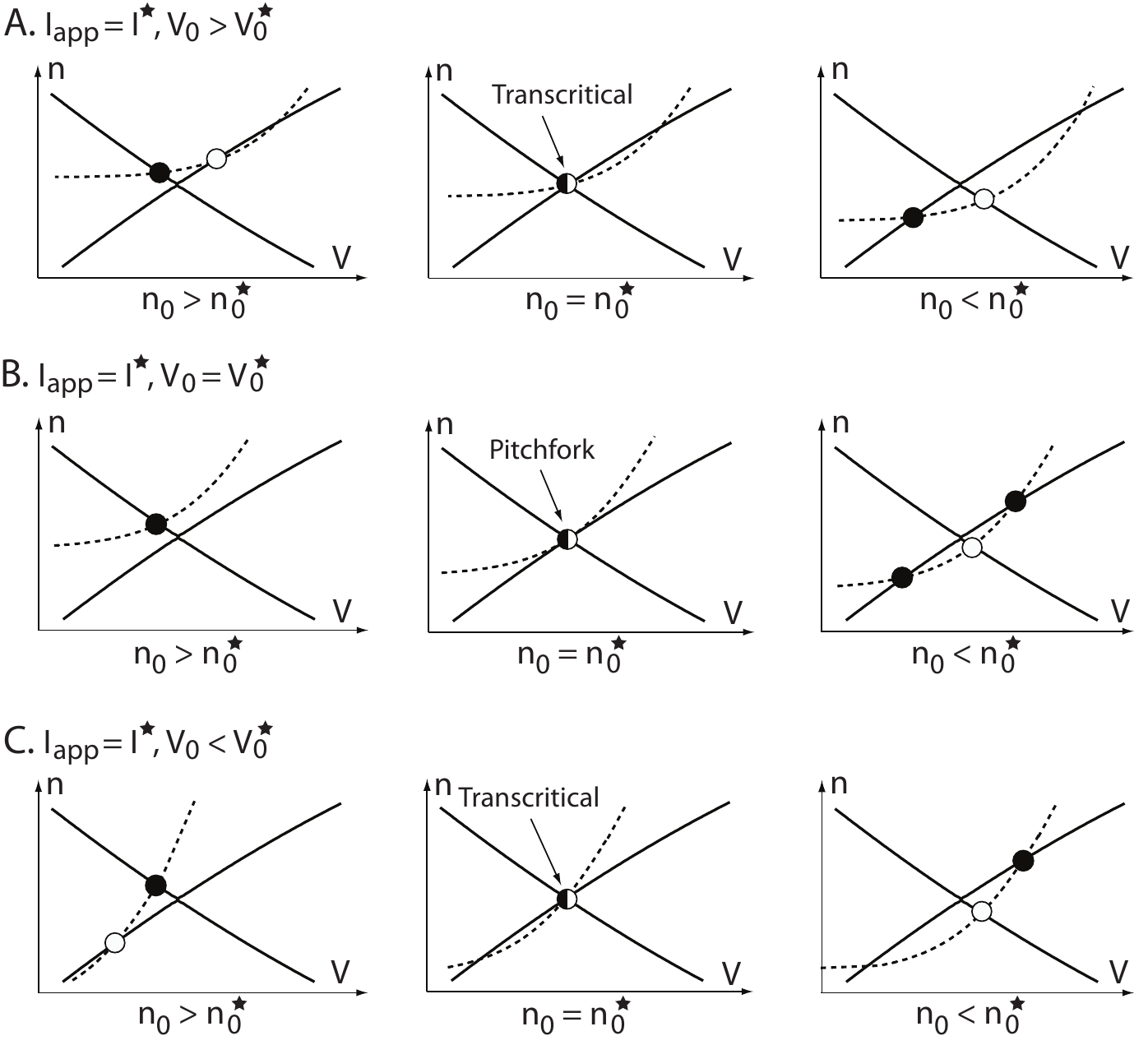}
\caption{Phase portrait of (\ref{EQ: mirrored FHN dynamics}) for $I_{app}=I^\star$, different values of $V_0$, and with $n_0$ as the bifurcation parameter. {\bf A,C}: When $V_0<V_0^\star$ or $V_0>V_0^\star$ as $n_0$ decreases below $n_0^\star(V_0)$ a stable and an unstable fixed points exchange their stability in a codimension 2 transcritical singularity. {\bf B:} When $V_0=V_0^\star$ the bifurcation degenerates in codimension 3 pitchfork at which a stable fixed point splits in two stable fixed points (outer) and an unstable fixed point (inner).}\label{FIG: trans and pitch}
\end{figure}
Note that we have to fix exactly two parameters ({\it i.e.} $I_{app}$ and $\delta_0$) to obtain (\ref{EQ: trans def conds}), which confirms that bifurcation described by (\ref{EQ: trans def conds}) has codimension 2.

\subsubsection*{Codimension 3 pitchfork bifurcation for $I_{app}=I^\star$, $n_0=n_0^\star(V_0)$, $V_0= V_0^\star$}

Adding the condition $V_0= V_0^\star$ (Item 3.), we obtain the extra degeneracy condition
$$\frac{\partial^2\dot v}{\partial v^2}\left|_{\substack{v=\delta=0\\
I_{app}=I^\star\\
V_0\neq V_0^\star}}\right.=2\big(1-k(V_0)^2\big)=0, $$
and the associated bifurcation is in this case a codimension 3 pitchfork bifurcation (see e.g. \cite[Page 367]{SEYDEL94}). In this case, positive perturbation of $\delta_0$ leads to a unique stable fixed point, which split in three fixed points, the outer stable and the central unstable, at the bifurcation (see Figure \ref{FIG: trans and pitch} {\bf B}).

\subsection*{Unfolding in the plane $I_{app}=I^\star$}

The parameter chart in Figure \ref{FIG: TC PC} unfolds the pitchfork bifurcation in the plane $(V_0,n_0)$ for the fixed critical value $I^\star$. This bifurcation analysis reveals five qualitatively distinct regions denoted by I, II, III, IV, and V. The transition from Region I to Region II is through a saddle-node bifurcation at which the $n$-nullcline is tangent to the $V$-nullcline. See how the top right phase portrait is continuously deformed to the top left phase portrait in Figure \ref{FIG: TC PC}. The transition from Region I to Region IV is through a transcritical bifurcation at which a saddle and a node exchange their stability. See how the top right phase portrait is continuously deformed to the bottom right phase portrait in Figure \ref{FIG: TC PC}. A similar transition occur from Region IV to Region V. See how the bottom right phase portrait is continuously deformed to the bottom left phase portrait in Figure \ref{FIG: TC PC}. Finally, the transition from Region V to Region II is through a saddle-node bifurcation. See how the bottom left phase portrait is continuously deformed to the top left phase portrait in Figure \ref{FIG: TC PC}.

Region III in Figure \ref{FIG: TC PC} is illustrated for future reference in the next section, but it should not be differentiated from Region II in the plane $I=I^\star$, {\it i.e.} there is no bifurcation associated to the transition from Region II to Region III.

\begin{figure}
\centering
\includegraphics[width=0.8\textwidth]{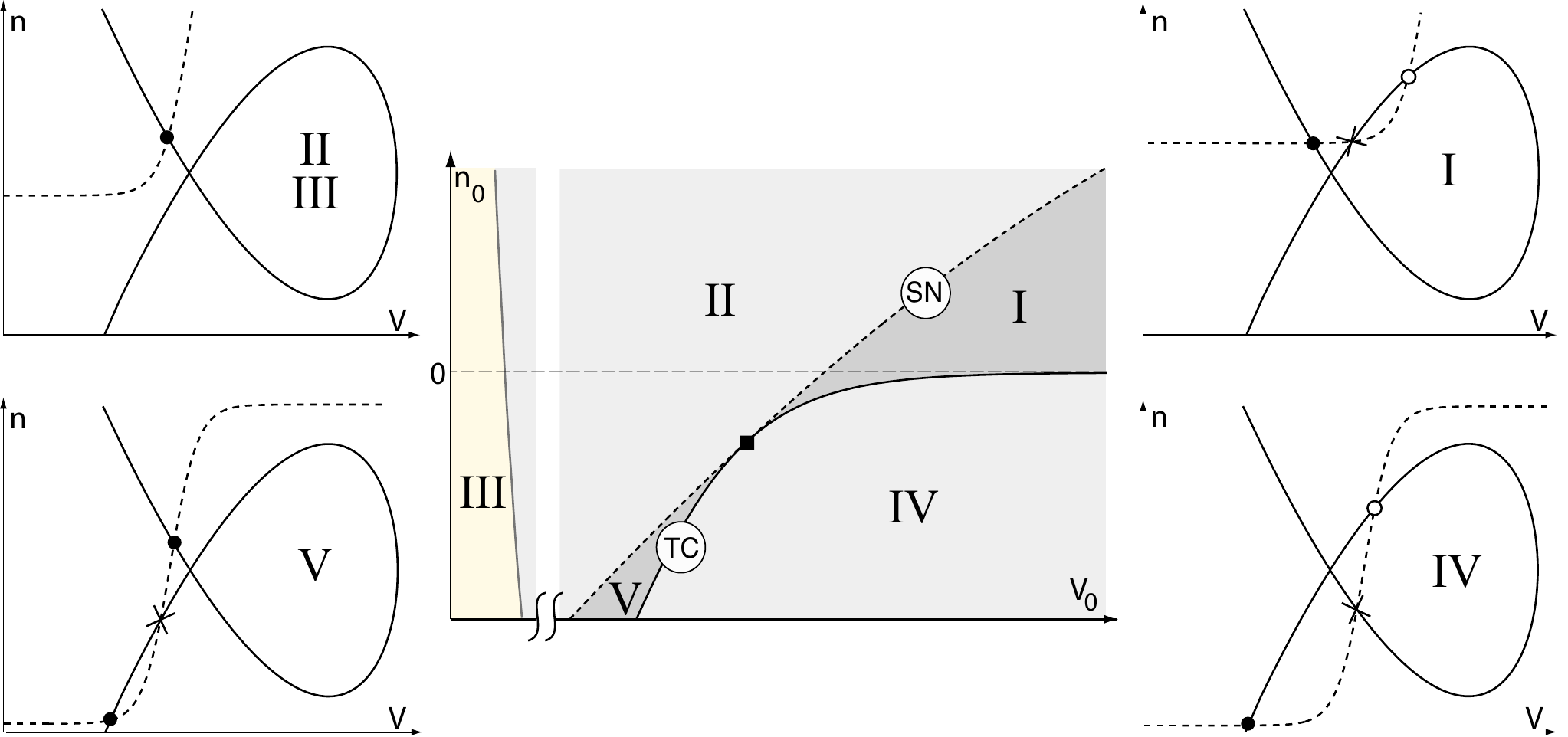}
\caption{Unfolding of the degenerate pitchfork bifurcation in the plane $I_{app}=I^\star$. Stable fixed points are depicted as filled circles, whereas unstable as circles. Saddle points are depicted as crosses. $\blacksquare$: pitchfork bifurcation. TC: transcritical bifurcation. SN: saddle-node bifurcation.}\label{FIG: TC PC}
\end{figure}

\subsection*{The pitchfork bifurcation organizes excitability}

The relevance of the unfolding in Figure \ref{FIG: TC PC} for excitability is that the different regions correspond to different types of excitability for $I_{app}>I^\star$. Regions I, II, and III correspond to Types I, II, and III excitability identified in the early work of Hodgkin \cite{hodgkin1948local} and extensively studied in the literature since then. Those are the only types of excitability that can be associated to purely competitive models ({\it i.e. $n_0>0$}) and they all have been studied in FitzHugh-Nagumo type phase portraits. They are briefly reviewed in Section \ref{SEC: known excitability types}. In contrast, Region IV and V correspond to new types of excitability that require the co-existence of competitive and cooperative ionic currents. They are studied in Sections \ref{SEC: novel excitability types}, \ref{SEC: Type IV and V phase portrait}. 
%and \ref{SEC: bist and burst in SP systems}.

Our analysis assumes a timescale separation $\epsilon\ll1$, reflecting the accepted strong separation between the \emph{fast} voltage dynamics and sodium activation kinetics and the remaining \emph{slow} gating kinetics. We focus on those bifurcations that persist in the singular limit $\epsilon\to0$.

Figure \ref{FIG: TC exc types} summarizes the different types of excitability studied in the next section and their main electrophysiological signatures.

\begin{figure}
\centering
\includegraphics[width=1.0\textwidth]{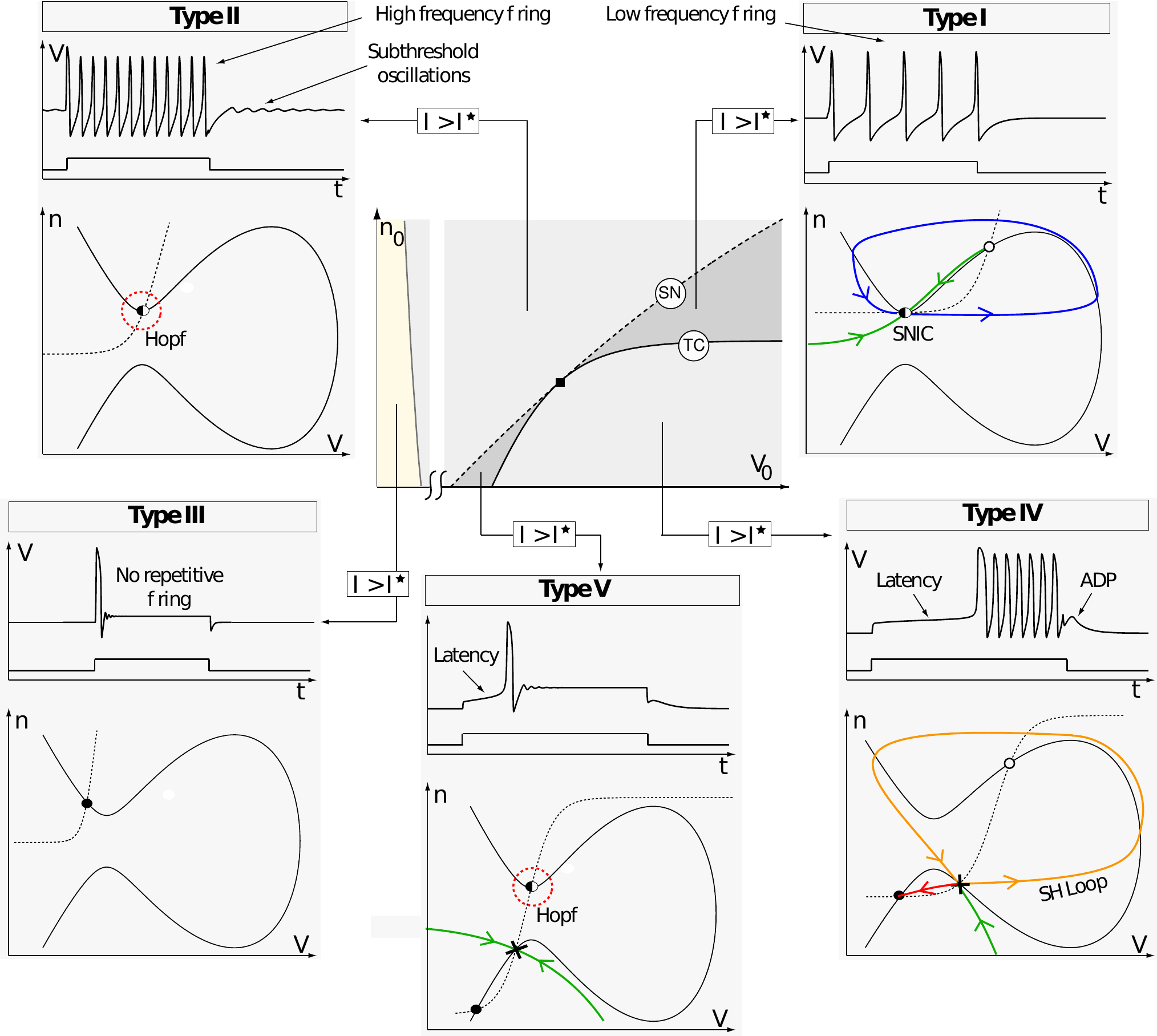}
\caption{Sketch of different types of excitable behaviors in the unfolding of the degenerate pitchfork bifurcation ($\blacksquare$ in the central parameter chart) for $I_{app}>I^\star$ and $\epsilon\ll1$. For each type, the typical voltage time course and the phase portrait are sketched. The abbreviation SNIC denotes the saddle-node on invariant circle bifurcation. Stable fixed points are depicted as filled circles, unstable as circles, saddles as cross, and bifurcations as half-filled circles. The stable manifold of saddle points is depicted in green. The center manifold of the SNIC bifurcation is depicted in blue. The saddle homoclinic loop in Type IV is depicted in orange.}
\label{FIG: TC exc types}
\end{figure}

\section{Three types of excitability in competitive models}
\label{SEC: known excitability types}

\subsection*{Type I (SNIC)}

Fixing the parameters $(V_0,n_0)$ in Region I of the parameter chart in Figure \ref{FIG: TC PC}, the node and the saddle approach each other as the applied current $I_{app}>I^\star$ increases and eventually collide in a saddle-node bifurcation at the critical value $I_{app}=I_{SNIC}$, as depicted in Figure \ref{FIG: TC exc types} (top right). Under the timescale separation assumption ({\it i.e.} $\epsilon\ll1$), the center manifold of the bifurcation forms a homoclinic loop, as sketched in the figure. This bifurcation is commonly identified as saddle-node on invariant circle (SNIC). As $I_{app}$ is further increased, the fixed point disappears and the system generates a periodic train of action potentials.\\

\noindent The excitability properties of model (\ref{EQ: mirrored FHN dynamics}) near a SNIC bifurcation are commonly referred to as Type~I excitability (see e.g. \cite[Section 3.4.4]{ERTE10}, \cite{Rinzel:1989:ANE:94605.94613}, and \cite[Section 7.1.3]{IZHIKEVICH2007} and references therein). The main properties associated to Type I excitability are as follows: 
\begin{itemize}
\item[] {\bf All-or-none spike.} The model has a well defined threshold ({\it i.e.} $I_{app}=I_{SNIC}$) to generate action potentials. Above the threshold, the amplitude of the action potentials does not depend on the stimulus intensity.
\item[] {\bf Low frequency spiking.} The frequency of the limit cycle decreases to zero as $I_{app}\searrow I_{SNIC}$. From a computational point of view, this property permits to encode the stimulus intensity in the oscillation frequency.
\end{itemize} 

Examples of neurons exhibiting Type I excitability include: thalamo-cortical neurons with inactivated T-type calcium current (depolarized steady-state) \cite[Figure 3]{zhan1999current}, isolated axons from Carcinus maenas \cite[Class 1]{hodgkin1948local}, regular spiking neurons in somatosensory cortex \cite{tateno2004threshold}, molluscan neurons \cite{sessley2002evidence}.

\subsection*{Type II (Hopf)}

In Region II of the parameter chart in Figure \ref{FIG: TC PC}, the nullclines intersect only once at a stable fixed point. As $I_{app}$ is increased, this fixed point looses stability in a Hopf bifurcation at $I_{app}=I_{Hopf}$, as illustrated in Figure \ref{FIG: TC exc types} (top left). Above this critical input current, the system possesses a stable limit cycle surrounding the unstable fixed point.\\

\noindent Excitability properties associated to a Hopf bifurcation are well known and define Type II excitability (see e.g. \cite[Section 3.4.4]{ERTE10}, \cite{Rinzel:1989:ANE:94605.94613}, and \cite[Section 7.1.3]{IZHIKEVICH2007} and references therein). Fundamental properties of Type II excitable systems include:
\begin{itemize}
\item[] {\bf No threshold.} When the bifurcation is supercritical, the amplitude of the limit cycle decreases to zero as $I_{app}\searrow I_{Hopf}$, which makes it hard to define a threshold for the generation of an action potential. Canard trajectories \cite{wechselberger2007canards} are sometime considered as ``soft'' threshold manifold between small and large amplitude action potential.
\item[] {\bf Subthreshold oscillations.} When the applied current is slightly below the bifurcation values ({\it i.e.} $I_{app}\lesssim I_{Hopf}$) the system trajectory relaxes to the fixed point with damped oscillations at the natural frequency of the Hopf bifurcation ({\it i.e.} the imaginary part of the eigenvalues at the bifurcation).
\item[] {\bf No low frequency firing.} The oscillation frequency is (almost) independent of the injected current and is equal to or larger than the natural frequency of the Hopf bifurcation.
%\item[] {\bf No spike latency.} The onset of spiking coincides with the onset of the stimulus.
\item[] {\bf Frequency preference.} Trains of small ($<I_{Hopf}$) amplitude inputs can induce spike if the intra stimulus frequency is resonant with the natural frequency of the Hopf bifurcation. This phenomenon is tightly linked to the presence of subthreshold oscillations and permits to detect the presence of resonant harmonics in the stimulus.
\item[] {\bf Post-inhibitory spike.} Transient negative current can induce an action potential in Type II excitable systems.
\end{itemize} 

Examples of neurons exhibiting Type II excitability include: isolated axons from Carcinus maenas \cite[Class 2]{hodgkin1948local}, fast spiking neurons in somatosensory cortex \cite{tateno2004threshold}, alpha moto-neurons \cite{Messina197657}.

\subsection*{Type III}

Type III excitability was only recently studied \cite{Gai01122009}. It can be thought as a less excitable variant of Type II excitability. In Region III of Figure \ref{FIG: TC PC}, the half-activation voltage $V_0$ is so negative that the stable focus never looses its stability as the applied current is increased. Nevertheless, the model is still excitable. For instance, as depicted in Figure \ref{FIG: TC exc types} (bottom left), a positive current step instantaneously shifts the stable fixed point upright and an originally resting trajectory is attracted toward the right branch of the $V$-nullcline before relaxing back to rest. On the contrary, if the applied current varies slowly no action potential is generated.\\

\noindent Particular neurocomputational properties of Type III excitability have recently been highlighted in \cite{Gai01122009}:
\begin{itemize}
\item[] {\bf Slope detection.} Because a brutal variation of the applied current is necessary to excite the model, Type III excitable neurons acts as slope detectors with a high temporal precision
\item[] {\bf Slope based stochastic resonance.} In the presence of noise, Type III excitable models are most sensitive to the stimulus slope and frequency, rather than to its amplitude. The associated stochastic resonance phenomenon (slope based stochastic resonance) exhibits distinctly different filtering properties with respect to the classical stochastic resonance in Type I/II excitable models.
\end{itemize}

Examples of neurons exhibiting Type III excitability include: squid giant axons (revised model) \cite{clay2008simple}, auditory brain stem \cite{Gai01122009}, isolated axons from Carcinus maenas \cite[Class 3]{hodgkin1948local}.

\section{Two novel types of excitability in non-competitive models}
\label{SEC: novel excitability types}

\subsection*{Type IV (singularly perturbed saddle-homoclinic)}

As illustrated in the bottom right phase portrait of Figure \ref{FIG: TC PC}, Type IV excitability is the first excitability type that involves the ``mirrored'' shape of the voltage nullcline in (\ref{EQ: mirrored FHN dynamics}): the stable node and the saddle lie on the lower cooperative branch of the $V$-nullcline. In particular, the hyperpolarized stable steady state of Type IV excitable models lies in the cooperative region of the phase portrait, {\it i.e.} where
\begin{equation*}
\frac{\partial\dot V}{\partial n}\frac{\partial\dot n}{\partial V}>0.
\end{equation*}
As a consequence, excitability properties of this phase portrait cannot be studied in FitzHugh-Nagumo like models.

Fixing the pair $(V_0,n_0)$ in Region IV of Figure \ref{FIG: TC PC}, we obtain the bifurcation diagram illustrated in Figure \ref{FIG: Type IV} together with the associated phase portraits. The stable node looses stability in a saddle-node bifurcation at the critical value $I_{app}=I_{SN}>I^\star$. For $I_{app}>I_{SN}$, the model possesses a stable limit cycle that attracts all solutions (but the unstable focus). The spiking limit cycle disappears in a (singularly perturbed) saddle-homoclinic bifurcation at $I_{app}=I_{SH}$. The stable node attracts all solutions (but those on the stable manifold of the saddle) for $I_{app}<I_{SH}$. Based on geometrical singular perturbations, we provide in Section \ref{SEC: Type IV and V phase portrait} a global phase portrait analysis of Type IV excitable systems.

\begin{figure}
\centering
\includegraphics[width=0.9\textwidth]{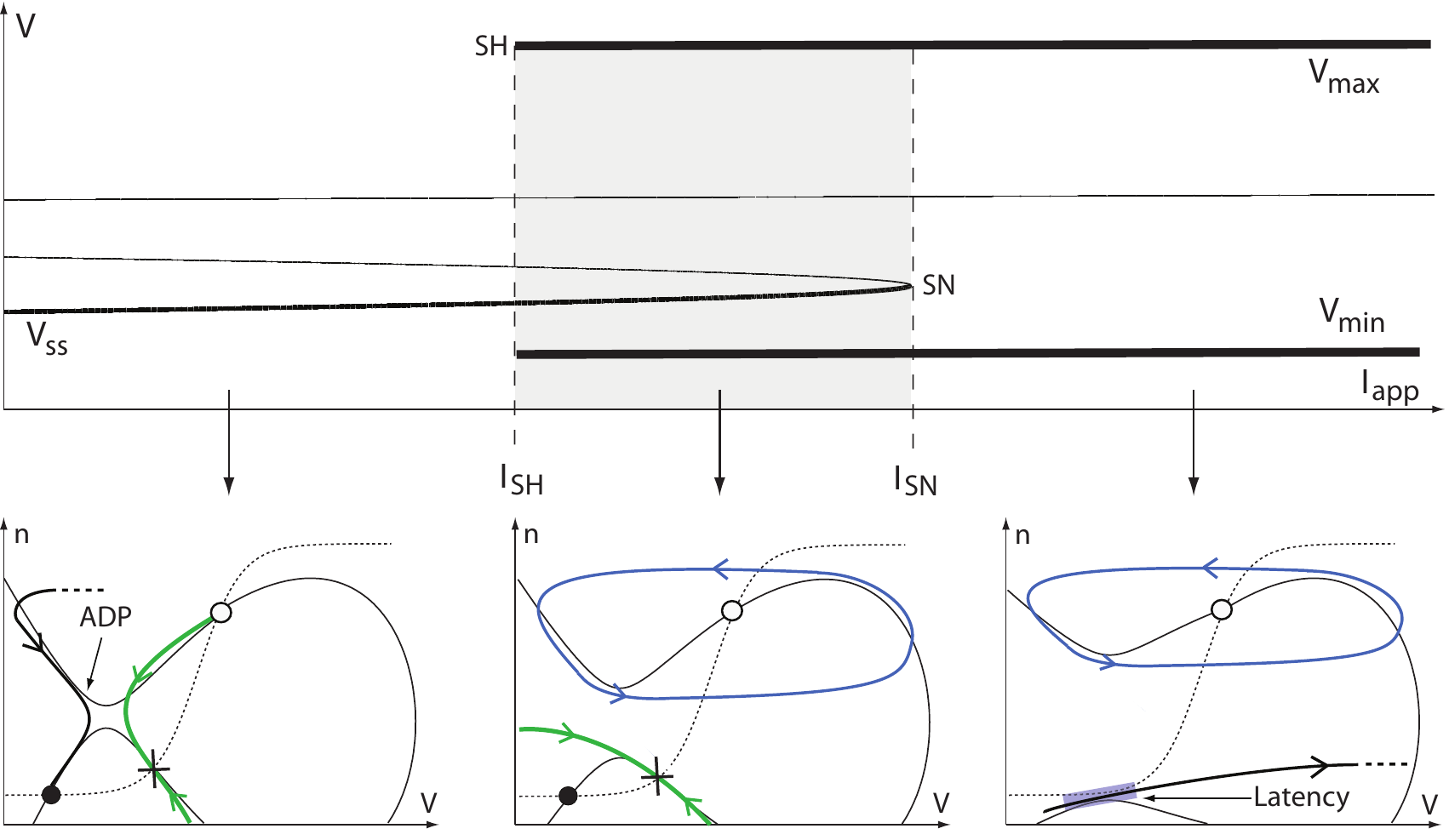}
\caption{Bifurcations in Types IV excitable systems. Top: bifurcation diagram. Stable fixed point are drawn as solid lines, unstable as thin lines, and stable limit cycles as think lines. SH denotes the saddle-homoclinic bifurcation, SN the saddle-node bifurcation. Bottom: phase portraits. The stable manifold of saddle points is depicted in green. Unstable fixed point are depicted as circles, saddles as crosses, and stable fixed points as filled circles. Limit cycle attractors are depicted in blue. Sample trajectories are depicted as black oriented lines.}\label{FIG: Type IV}
\end{figure}

\noindent The chief properties of Type IV excitable systems are:
\begin{itemize}
\item[] {\bf Bistability.} Type IV excitable models are bistable in the parameter range $I_{SH}<I_{app}<I_{SN}$: a limit cycle attractor coexists with a stable fixed point.
\item[] {\bf Spike-latency.} When the applied current is abruptly increased slightly above $I_{SN}$, the trajectory necessarily travels the narrow region between the nullclines. Since the vector field is small in that region, the first action potential is fired with a large latency.
\item[] {\bf After-depolarization potential.} After the limit cycle has disappeared in the homoclinic bifurcation, the trajectory converges to rest by following the attractive branches of the voltage nullcline, thus generating robust ADPs (see \cite{DRFRSESE_PLoS_sub}).
\end{itemize}
We stress that bistability, spike-latency, and ADPs are all direct consequences of the presence of a saddle point on the cooperative branch of the voltage nullcline: the basin of attraction of the stable node and the limit cycle are separated by the stable manifold of the saddle; spike-latency reveals the ``ghost'' of the center manifold of the saddle-node bifurcation; ADPs are generated along the hyperbolic invariant structure provided by the saddle stable and unstable manifolds (see Section \ref{SEC: Type IV and V phase portrait} below).

Examples of neurons exhibiting Type IV excitability include: subthalamic nucleus neurons \cite{Hallworth20082003}, thalamo-cortical reticular and relay neurons with deinactivated T-type calcium current (hyperpolarized state) \cite{huguenard1992novel,mccormick1997sleep}, dopaminergic neurons \cite{johnson1992burst,grace1984control}, superficial pyramidal neurons \cite{gray1996chattering}.

\subsection*{Type V (saddle-saddle)}

Type V excitability relates to Type IV as Type III does to Type II: similarly to Type IV, the hyperpolarized stable steady state of Type V excitable models lies in cooperative
region of the phase plane. The distinct feature of the bottom left phase portrait of Figure \ref{FIG: TC PC} is the co-existence of two stable fixed points, a ``down-state'' and an ``up-state''. The saddle stable manifold separates the two attractors. As $I_{app}$ is decreased below $I^\star$, the up-state eventually looses its stability in a saddle-node bifurcation at $I_{SN,up}$, leaving the down-state globally asymptotically stable. Similarly, as $I_{app}$ is increased above $I^\star$, the down-state eventually disappears in a saddle-node bifurcation at $I=I_{SN,down}$. But the up-state itself eventually looses its stability in a Hopf bifurcation at $I_{app}=I_{Hopf}$. The Hopf bifurcation can either take place beyond the  bistable range $[I_{SN,up},I_{SN,down}]$, a situation illustrated in Figure \ref{FIG: Type Va}, or it can take place within the bistable range $[I_{SN,up},I_{SN,down}]$, in which case, depending on $I_{app}$, the stable down-state coexists with either a limit cycle attractor or a stable fixed point, a situation illustrated in Figure \ref{FIG: Type Vb}.

\begin{figure}
\centering
\includegraphics[width=0.9\textwidth]{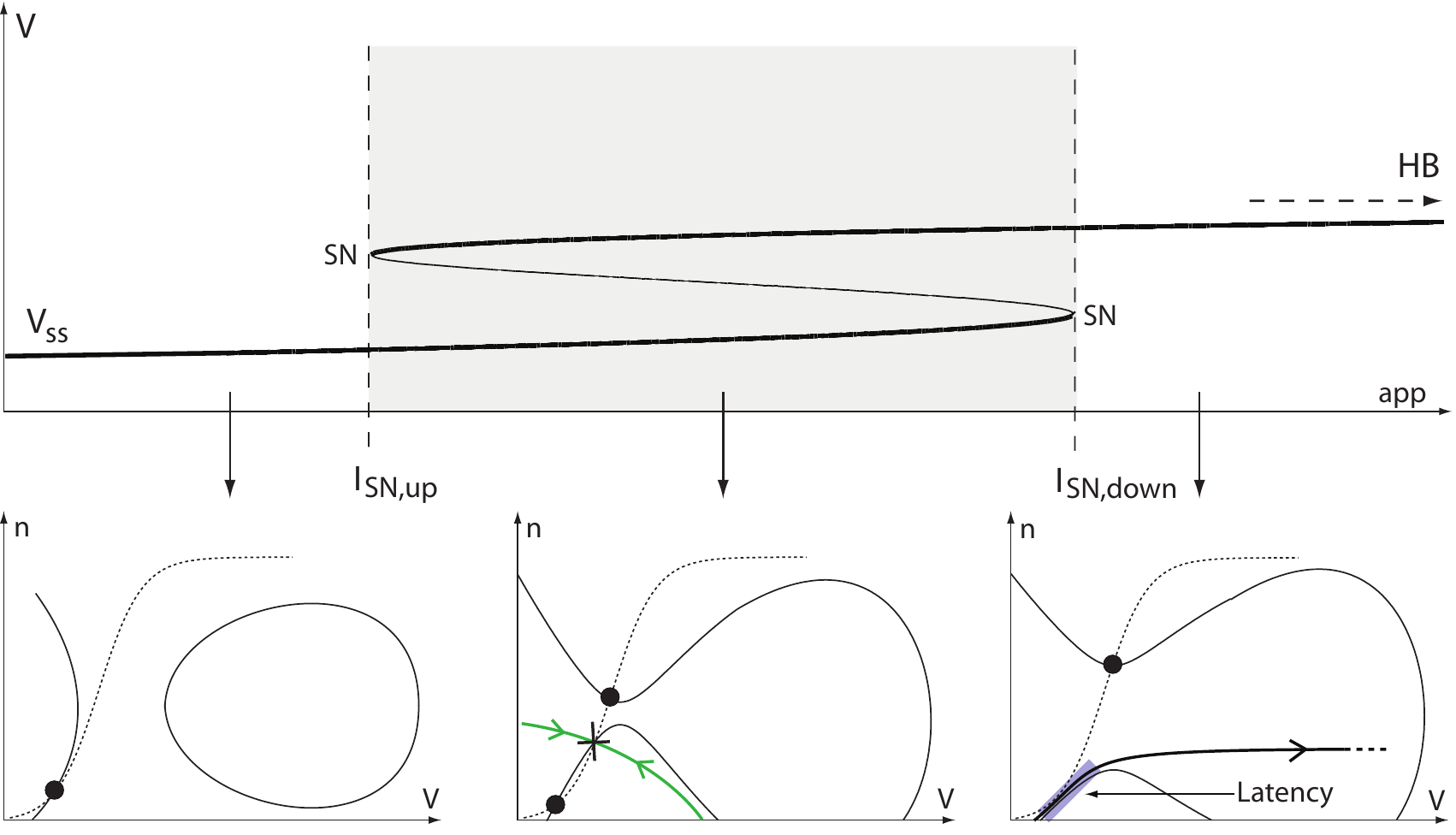}
\caption{Bifurcations in Types V excitable systems. Legend as in Figure \ref{FIG: Type IV}, except HB denoting the Hopf bifurcation.}\label{FIG: Type Va}
\end{figure}
\begin{figure}
\centering
\includegraphics[width=0.9\textwidth]{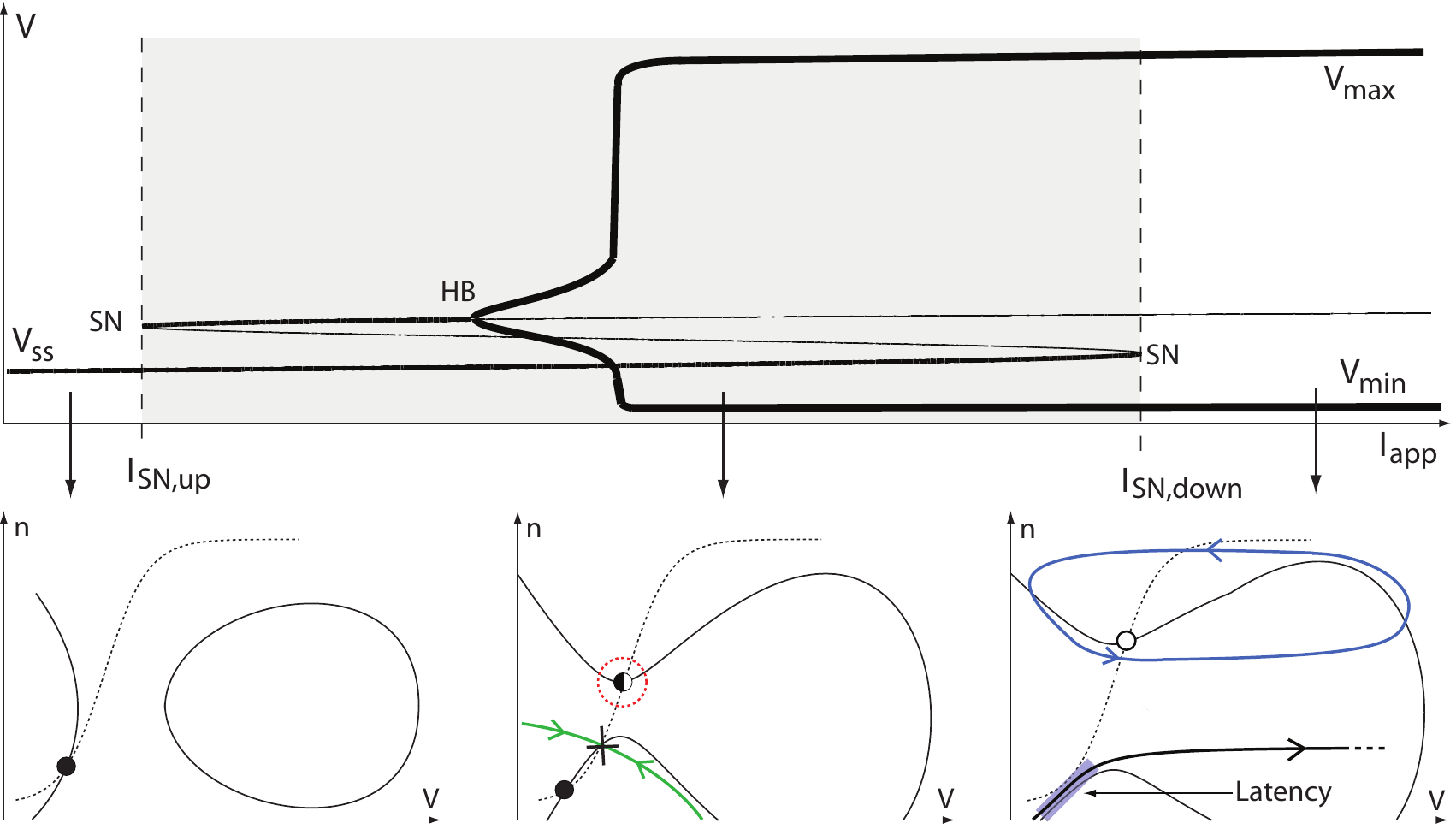}
\caption{Bifurcations in Types V excitable systems. Legend as in Figure \ref{FIG: Type IV}, except HB denoting the Hopf bifurcation. The same bifurcation is depicted in the phase portrait a half-filled circle.}\label{FIG: Type Vb}
\end{figure}

\noindent The main properties of Type V excitable models are similar to those of Type IV and can be summarized as follows:
\begin{itemize}
\item[] {\bf Bistability} Type V excitable models are bistable in the range $I_{SN,up}<I_{app}<I_{SN,down}$: a stable down-state coexists with an up-state attractor that can be either a stable fixed point or a stable limit cycle.
\item[] {\bf Spike latency.} Similarly to Type IV, the down-state looses stability in a saddle-node bifurcation on the lower (cooperative) branch of $V$-nullcline, leading to a long latency before the convergence to the up-state.
\item[] {\bf Plateau potentials.} The up-state having a higher voltage with respect to the down-state, the transition between the two gives rise to plateau potentials either with or without spikes.
\end{itemize}

Examples of neurons exhibiting Type V excitability include: olfactory bulb mitral cells \cite{heyward2001membrane} and striatal medium spiny neurons \cite{wilson1981spontaneous}. More examples are listed at Scholarpedia journal article \cite{wilson2008up}.

\section{Singularly perturbed global phase portrait analysis of Types IV and V excitable models}
\label{SEC: Type IV and V phase portrait}

Under timescale separation assumption ({\it i.e.} $0<\epsilon\ll1$), Types IV and V excitable models exhibit four distinct signatures with respect to Types I and II: {\it (i)} the existence of a saddle-homoclinic bifurcation and {\it (ii)} ADPs in Type IV; {\it (iii)} bistability and {\it (iv)} spike latency in both Types IV and V. A global phase portrait analysis based on geometrical singular perturbations provides an analytical explanation for the occurrence of these signatures \emph{solely} in Types IV and V.

\subsection{Singularly perturbed saddle-homoclinic bifurcation and robust ADP generation in Type IV excitable models}

In this section we prove the existence of the saddle-homoclinic loop in Type IV excitable systems. We then rely on this analysis to provide a qualitative picture of the ADP generation mechanism in these models. In the remainder of the section we assume that the pair $(V_0,n_0)$ lies in Region IV of Figure \ref{FIG: TC PC}.

We start by briefly recalling some basic results of geometrical singular perturbation theory, using (\ref{EQ: mirrored FHN dynamics}) as an explicit example. The interested reader will find in \cite{JONES95} an excellent introduction to the topic, and in \cite{KRSZ2001relax,KRSZ01,KRSZ01a} some recent extensions on which we rely for the forthcoming analysis. The time rescaling $\tau:=\epsilon t$ transforms (\ref{EQ: mirrored FHN dynamics}) into the equivalent system
\begin{subequations}\label{EQ: slow time scale}
\begin{eqnarray}
\epsilon\dot V&=&V-\frac{V^3}{3}-n^2+I_{app}\IEEEyessubnumber\\
\dot n&=&\epsilon(n_{\infty}(V-V_0)+n_0-n)\IEEEyessubnumber,
\end{eqnarray}
\end{subequations}
which describes the dynamics (\ref{EQ: mirrored FHN dynamics}) in the slow timescale $\tau$.
In the limit $\epsilon=0$, commonly referred to as the \emph{singular limit}, one obtains from (\ref{EQ: mirrored FHN dynamics}) and (\ref{EQ: slow time scale}) two new dynamical systems: the \emph{reduced dynamics}
\begin{subequations}\label{EQ: reduced problem}
\begin{eqnarray}
0&=&V-\frac{V^3}{3}-n^2+I_{app}\IEEEyessubnumber\\
\dot n&=&\epsilon(n_{\infty}(V-V_0)+n_0-n)\IEEEyessubnumber,
\end{eqnarray}
\end{subequations}
which evolve in the slow timescale $\tau$, and the \emph{layer dynamics}
\begin{subequations}\label{EQ: layer problem}
\begin{eqnarray}
\dot V&=&V-\frac{V^3}{3}-n^2+I_{app}\IEEEyessubnumber\\
\dot n&=&0\IEEEyessubnumber,
\end{eqnarray}
\end{subequations}
which evolve in the fast timescale $t$. Figure \ref{FIG: geometrical singular perturbation theory full}(a) depicts the fast-slow dynamics (\ref{EQ: reduced problem}),(\ref{EQ: layer problem}). The main idea behind geometrical singular perturbation theory is to combine the analysis of the reduced and layer dynamics to derive conclusions about the behavior of the nominal system, {\it i.e.} with $\epsilon>0$.
\begin{figure}
\centering
%\subfigure[][]{
%\input{./geometrical singular perturbation theory_Figs/slow_fast_dyn.pdf_t}
%}
%\subfigure[][]{
%\input{./geometrical singular perturbation theory_Figs/continuation_3.pdf_t}
%}
%\subfigure[][]{
%\input{./geometrical singular perturbation theory_Figs/continuation_2.pdf_t}
%}
%\subfigure[][]{
%\input{./geometrical singular perturbation theory_Figs/continuation_1.pdf_t}
%}
\includegraphics[width=0.9\textwidth]{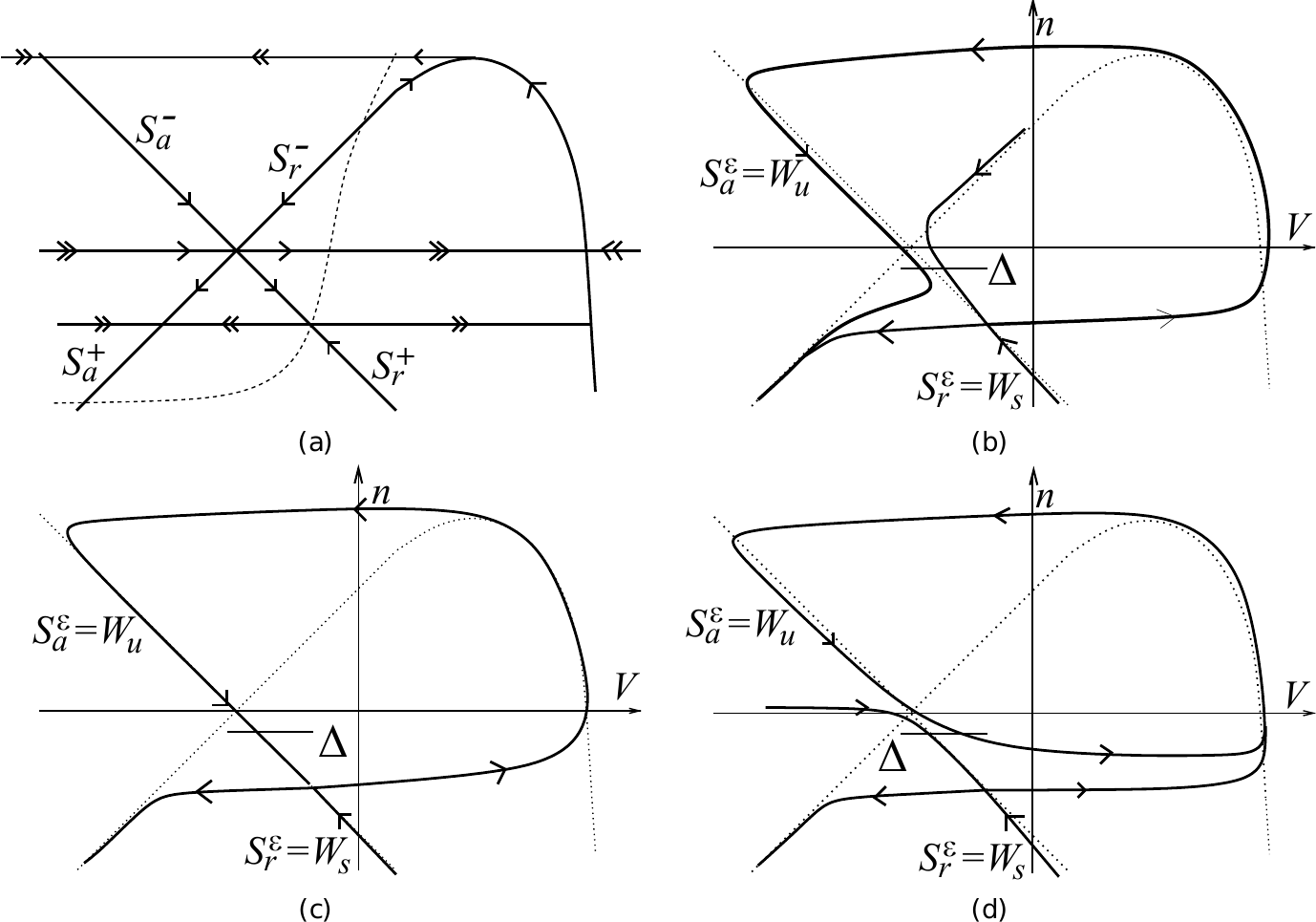}
\caption{{\bf Phase-portrait of (\ref{EQ: mirrored FHN dynamics}) for $(n_0,V_0)$ in Region IV of the parameter chart in Figure \ref{FIG: TC PC} and different values of $\epsilon$ and $I$.} (a):  fast-slow dynamics of (\ref{EQ: mirrored FHN dynamics}) for $\epsilon=0$ and $I=I^\star$. The attractive (resp. repelling) branches of the critical manifold $S_0$ above/below the transcritical singularity are denoted by $S_a^+/S_a^-$ (resp. $S_r^+/S_r^-$). (b): Continuation of the slow attractive $S_{a}^\epsilon$ and repelling $S_r^\epsilon$ manifolds for $\epsilon>0$ and $I_{app}<I^\star+I_c(\sqrt{\epsilon})$, where $I_c(\sqrt{\epsilon})$ is defined as in Theorem \ref{THM: KRSZ01 thm reformulation}. (c): Continuation of $S_{a}^\epsilon$ and $S_r^\epsilon$ for $\epsilon>0$ and $I=I^\star+I_c(\sqrt{\epsilon})$. (d): Continuation of $S_{a}^\epsilon$ and $S_r^\epsilon$ for $\epsilon>0$ and $I>I^\star+I_c(\sqrt{\epsilon})$.}\label{FIG: geometrical singular perturbation theory full}
\end{figure}

The reduced dynamics (\ref{EQ: reduced problem}) is a dynamical system on the set $$S_0:=\left\{(V,n)\in\RR^2:\ V-\frac{V^3}{3}-n^2+I_{app}=0\right\},$$ usually called the \emph{critical manifold}. The points in $S_0$ are indeed critical points of the layer dynamics (\ref{EQ: layer problem}). More precisely, portions of $S_0$ on which $\frac{\partial\dot V}{\partial V}$ is non-vanishing are normally hyperbolic invariant manifolds of equilibria of the layer dynamics, whose stability is determined by the sign of $\frac{\partial\dot V}{\partial V}$. Conversely, points in $S_0$ where $\frac{\partial\dot V}{\partial V}=0$ constitute degenerate equilibria. In particular, the layer dynamics (\ref{EQ: layer problem}) exhibits, for $I_{app}=I^\star$, two degenerate equilibria\footnote{A third degenerate equilibria is at the minimum of the $V$-nullcline that is specular to the maximum with respect to the line $v=0$, but it plays no dynamical role and is not considered here.}. As depicted in Figure \ref{FIG: geometrical singular perturbation theory full}(a) they are given by the self-intersection of the $V$-nullcline, which is the transcritical organizing center in Figure \ref{FIG: trans organizing center}, and by the fold singularity at the maximum of the upper branch of the $V$-nullcline.

The basic result of geometrical singular perturbation theory, due to Fennichel \cite{FENICHEL79}, is that, for $\epsilon$ sufficiently small, non-degenerate portions of $S_0$ persist as nearby normally hyperbolic locally invariant manifolds $S^\epsilon$ of (\ref{EQ: mirrored FHN dynamics}). More precisely, the \emph{slow manifold} $S^\epsilon$ lies in a neighborhood of $S_0$ of radius $\mathcal O(\epsilon)$. The dynamics on $S^\epsilon$ is a small perturbation of the reduced dynamics (\ref{EQ: reduced problem}). We point out that $S^\epsilon$ may not be unique, but is determined only up to $\mathcal O(e^{-c/\epsilon})$, for some $c>0$. That is, two different choices of $S^\epsilon$ are exponentially close (in $\epsilon$) one to the other. Since the presented results are independent of the particular $S^\epsilon$ considered, we let this choice be arbitrary. The trajectories of the layer dynamics perturb to a stable and an unstable invariant foliations with basis $S^\epsilon$.

The analysis near degenerate points is more delicate. Only recently some works have treated this problem in its full generality for different types of singularities \cite{KRSZ2001relax,KRSZ01,KRSZ01a}. Figure \ref{FIG: geometrical singular perturbation theory full} (b),(c),(d) sketch the extension of the attractive slow manifold $S_a^\epsilon$ after the fold point, and the three possible ways in which $S_a^\epsilon$ and the repelling slow manifold $S_r^\epsilon$ can continue after the transcritical singularity, depending on the injected current.

The result depicted in Figure \ref{FIG: geometrical singular perturbation theory full} relies on the following analysis, adapted from \cite{KRSZ01}.

Let $\Delta:=\{(V,n)\in\mathbb R^2:\ V_-\leq V \leq V_+,\ n=\rho\}$, be the section depicted in Figure \ref{FIG: geometrical singular perturbation theory full}, where $\rho<0$ and $|\rho|$ is sufficiently small, and $V_-,V_+$ are such that $\Delta\cap S_r^-\neq\emptyset$.
For a given $\epsilon>0$, let $q_{a,\epsilon}:=\Delta\cap S_a^\epsilon$ and $q_{r,\epsilon}:=\Delta\cap S^\epsilon_r$ be the intersections, whenever they exist, of respectively the attractive and repelling invariant submanifolds $S_a^\epsilon$ and $S_r^\epsilon$ with the section $\Delta$.
The following theorem reformulates in a compact way the discussion contained in Remark 2.2 and Section 3 of \cite{KRSZ01}\footnote{The first author is thankful to Prof. Szmolyan for his useful comments.} for the system with inputs (\ref{EQ: mirrored FHN dynamics}).

\begin{theorem}[Adapted from \cite{KRSZ01}]\label{THM: KRSZ01 thm reformulation}
Consider the system (\ref{EQ: mirrored FHN dynamics}). Then there exists $\epsilon_0>0$ and a smooth function $I_c(\sqrt{\epsilon})$, defined on $[0,\epsilon_0]$ and satisfying $I_c(0)=0$, such that, for all $\epsilon\in(0,\epsilon_0]$, the following assertions hold
\begin{enumerate}
\item $q_{a,\epsilon}=q_{r,\epsilon}$ if and only if $I_{app}=I^\star+I_c(\sqrt{\epsilon})$
\item there exists an open interval $A\ni I^\star+I_c(\sqrt{\epsilon})$, such that, for all $I\in A$, it holds that $\Delta\cap S_a^\epsilon\neq\emptyset$, $\Delta\cap S_r^\epsilon\neq\emptyset$, and
$$\frac{\partial}{\partial I_{app}}\left(q_{a,\epsilon}-q_{r,\epsilon} \right)>0. $$
\end{enumerate}
\end{theorem}
Figure \ref{FIG: geometrical singular perturbation theory full} illustrates this result.
\begin{remark}
The function $I_c(\sqrt{\epsilon})$ is related to the function $\lambda_c(\sqrt{\epsilon})$ defined in \cite[Remark 2.2]{KRSZ01} by $I_c(\sqrt{\epsilon}):=\epsilon\lambda_c(\sqrt{\epsilon})$. Similarly, given $\epsilon>0$, the parameter $I_{app}$ appearing in Theorem \ref{THM: KRSZ01 thm reformulation} is just the re-scaling $I_{app}=\epsilon\lambda+I^\star$ of the parameter $\lambda$ appearing in \cite[Remark 2.2 and Sections 3]{KRSZ01}.
\end{remark}

Theorem \ref{THM: KRSZ01 thm reformulation} implies the existence of the saddle-homoclinic bifurcation in the mirrored FitzHugh-Nagumo model (\ref{EQ: mirrored FHN dynamics}) with parameters $(V_0,n_0)$ belonging to Region IV of Figure \ref{FIG: TC PC}. In this case, as illustrated in Figure \ref{FIG: geometrical singular perturbation theory full}(b,c,d), the slow attractive manifold $S_a^\epsilon$ (resp. slow repelling manifold $S_r^\epsilon$) coincides with the unstable manifold $W_u$ (resp. stable manifold $W_s$) of the saddle point, as it can be proved via simple qualitative arguments. Thus, for $I_{app}<I^\star+I_c(\sqrt{\epsilon})$, the unstable manifold $W_u$ continues after the transcritical singularity on the left of $W_s$, toward the stable node. See Figure \ref{FIG: geometrical singular perturbation theory full}(b). For $I_{app}=I^\star+I_c(\sqrt{\epsilon})$, $W_u$ extends after the transcritical point to $W_s$, forming the saddle-homoclinic trajectory, as sketched in Figure \ref{FIG: TC exc types} and Figure \ref{FIG: geometrical singular perturbation theory full}(c). For $I_{app}>I^\star+I_c(\sqrt{\epsilon})$, the unstable manifold of the saddle $W_u$ continues after the transcritical singularity on the right of $W_s$, and spirals toward an exponentially stable limit cycle, whose existence can be proved with similar geometrical singular perturbation theory arguments (see for instance \cite{KRSZ2001relax}). This situation is the one depicted in Figure \ref{FIG: geometrical singular perturbation theory full}(d).

The existence of the saddle-homoclinic bifurcation near the singular limit $\epsilon=0$ is a direct consequence of the existence of a singular connection between the stable and unstable manifold of the saddle, as sketched in Figure \ref{FIG: singular st-unst conncection}.
%In the fast slow dynamics (\ref{EQ: reduced problem}),(\ref{EQ: layer problem}), the unstable manifold of the saddle degenerate into the critical line passing by the saddle. At the right branch of the $V$-nullcline it slides up to the maximum, where it jumps to to the left $V$-nullcline branch following the critical line passing by the maximum. This branch finally connects it to the (singular) stable manifold of the saddle across the transcritical singularity.
The existence of the saddle-homoclinic bifurcation can be seen as the persistence of this singular saddle-homoclinic loop in the non-singular dynamics. The same loop being absent in neuron models with an $N$-shaped nullcline, as discussed below, we identify with the singularly perturbed saddle-homoclinic bifurcation sketched in Figure \ref{FIG: TC exc types} (bottom right) as the signature of Type IV excitable models.
\begin{figure}
\centering
\includegraphics[width=0.5\textwidth]{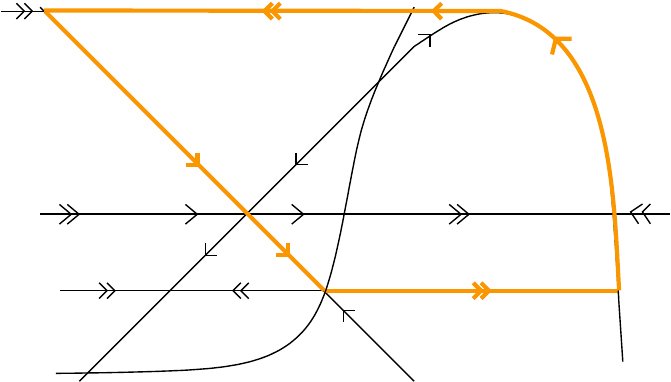}
\caption{Singularly perturbed non-trivial saddle-homoclinic connection near the transcritical singularity. This loop persists in the non-singular limit.}\label{FIG: singular st-unst conncection}
\end{figure}

When $I_{app}<I^\star+I_c(\sqrt{\epsilon})$ the above analysis ensures the robust generation of ADPs. The trajectory relaxation to the stable fixed point is indeed guided by the normally hyperbolic attractive manifold $S_a^\varepsilon$. As illustrated in Figure \ref{FIG: geometrical singular perturbation theory full}(b), the voltage is not monotone as the trajectory slides along this manifold, corresponding to an ADP. The normally hyperbolicity of $S_a^\varepsilon$ ensures that the ADP generation is robust to external perturbations. Some consequences for the modeling of neurons exhibiting ADPs can be found in \cite{DRFRSESE_PLoS_sub}.

\subsubsection*{Absence of singularly perturbed saddle-homoclinic bifurcation and ADPs in competitive models}\label{SEC: no N-shaped sadhom}

When the stable fixed point and the saddle belong to the upper $N$-shaped competitive branch of the voltage nullcline ({\it i.e.} Type I excitability), the model cannot exhibit saddle-homoclinic bifurcation.
%Saddle-homoclinic bursting is typical in a large variety of neurons \af{[REFS]}. However, in the presence of timescale separation between the voltage and the recovery variable (as it is the case, at least for neurons), then this type of bursting cannot be observed in planar neuron models with only an N-shaped $V$-nullcline branch (as the original Fitz-Hugh model) or, similarly, when only this branch of the $V$-nullcline is contained in the physiological region $S_0$, defined in (\ref{EQ: invariant physiological strip}), as in Figure \ref{FIG: n0 role and physio region}A. More precisely, we claim that, when the are a node, a saddle, and unstable focus belonging to the upper branch of the $V$-nullcline as in Figure \ref{FIG: no sad-hom connection}, if $\epsilon>0$ is sufficiently small, then there are no saddle-homoclinic orbits, indipendently of the injected current $I_{app}$. This claim can easily be proved by geometrical singular perturbation (geometrical singular perturbation theory) arguments.
As sketched in Figure \ref{FIG: no sad-hom connection}, in the fast-slow dynamics there are no nontrivial saddle-homoclinic connections. By persistence arguments \cite{FENICHEL79,KRSZ2001relax}, this implies that the nontrivial intersection of the stable and unstable manifolds of the saddle is empty also away from the singular limit.

The prediction of singular perturbation theory of course does not contradict the existence of a saddle-homoclinic bifurcation in competitive models, it only precludes it for sufficiently small values of $\epsilon$. A well-known saddle-homoclinic bifurcation is described in Morris-Lecar model \cite{morris1981voltage} for the barnacle giant muscle fiber. The model is purely competitive but a saddle-homoclinic bifurcation is possible around the N-shaped voltage nullcline (see for instance \cite[Section 3.4.3]{ERTE10}). However, such a bifurcation cannot persist with a strong timescale separation, which suggests that it might be less relevant in the context of neuronal modeling.

\begin{figure}[h!]
\centering
\includegraphics[width=0.9\textwidth]{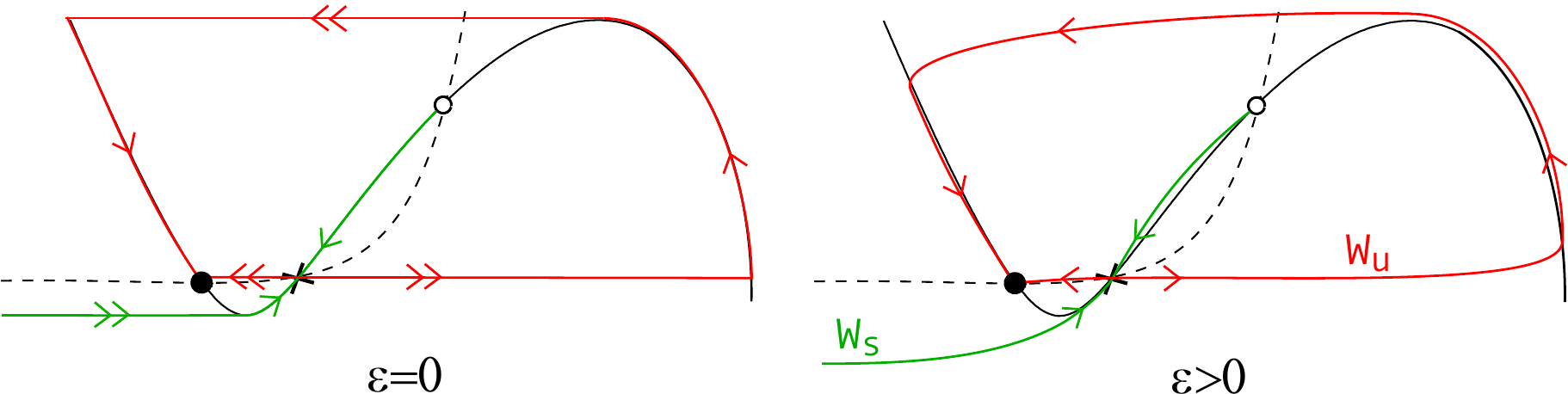}
\caption{Absence of a non-trivial singularly perturbed saddle-homoclinic connection around the $N$-shaped $V$-nullcline branch. Such connection is absent also for sufficiently small $\epsilon>0$.}\label{FIG: no sad-hom connection}
\end{figure}

The generation of ADPs must also be excluded when the stable fixed point belongs to the upper branch of the voltage nullcline. The relaxation to rest is indeed guided by the left attractive branch of the voltage nullcline, along which the voltage is monotone. ADPs can be generated within purely competitive models only by resorting to a non-physiological state reset mechanism \cite[Section 8.3]{IZHIKEVICH2007}.

\subsection{Singularly perturbed bistability in Type IV and V excitable systems}
\label{SEC: SP bistability}

The geometrical singular perturbation machinery introduced above can be used to prove the persistence of bistability near the singular limit $\epsilon=0$ for both Type IV and V excitable systems. The analysis is sketched in Figure \ref{FIG: SP bist}. The underlying key ingredient is the existence of singularly perturbed separatrix $W_s^0$ passing by the saddle point. This object persists in the non-singular limit as a normally hyperbolic saddle stable manifold $W_s^\epsilon$ that separates the stable down-state from the up-state attractor (see the bottom center plot in Figures \ref{FIG: Type IV}, \ref{FIG: Type Va}, and \ref{FIG: Type Vb}).
%The latter can be either a stable fixed point (see Figure \ref{FIG: Type Va} bottom center) or a limit cycle attractor surrounding the $N$-shaped competitive branch of the voltage nullcline (see Figure \ref{FIG: Type IV}). The transition between the two bistable regime is through a Hopf bifurcation near the minimum of the upper branch of the voltage nullcline (see Figure \ref{FIG: Type Vb} bottom center).
The phase portraits in the non-singular limit can all be deduced by the results in \cite{KRSZ2001relax,KRSZ01,KRSZ01a}.

\begin{figure}
\centering
\includegraphics[width=0.8\textwidth]{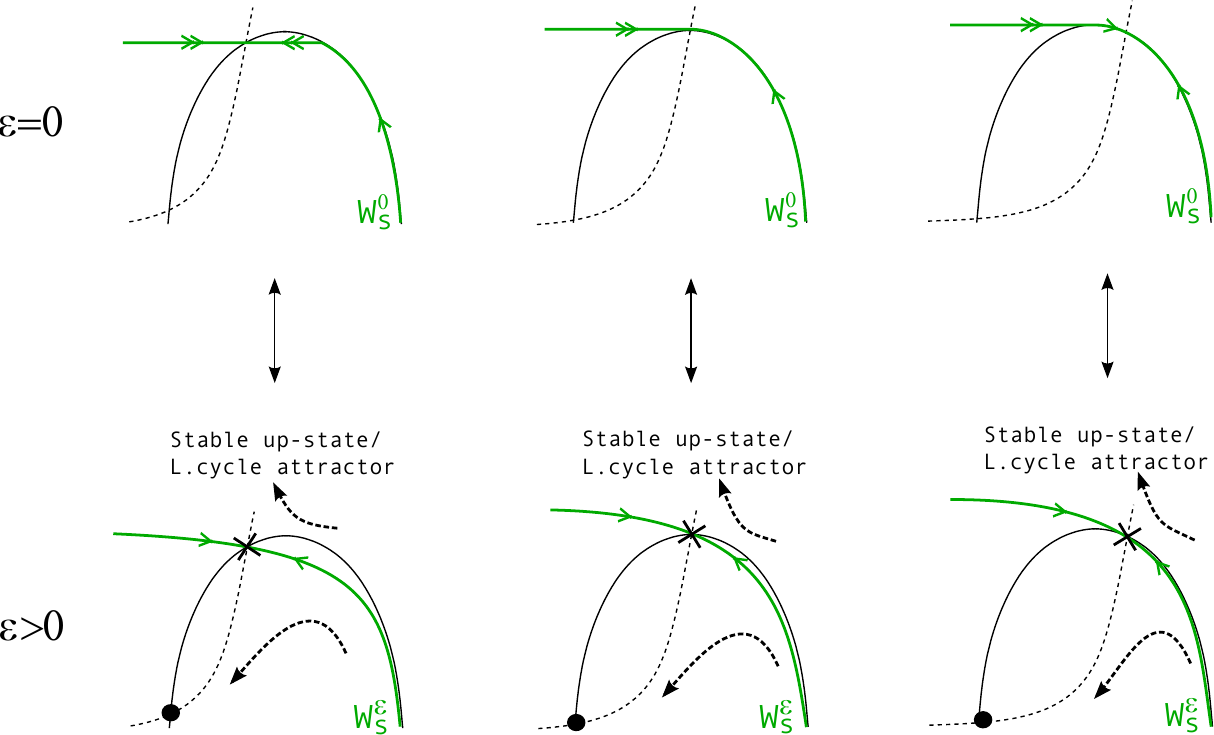}
\caption{Existence of a singularly perturbed saddle separatrix $W_s^0$ on the lower branch of the voltage nullcline in Type IV and V excitable systems for $I>I^\star$ and different nullcline intersections. Such object persists in the non-singular limit as the saddle stable manifold $W_s^\epsilon$ which separates the basin of attraction of the lower and upper attractors.}\label{FIG: SP bist}
\end{figure}
\subsubsection*{Absence of singularly perturbed bistability in competitive models}
As opposed to Types IV and V, bistability is not observed in Types I, II, and III excitable models near the singular limit $\epsilon=0$. For Type I this can be deduced by the absence of a singularly perturbed separatrix (cf. Figure \ref{FIG: no sad-hom connection}) and standard persistence arguments. For Type II, one can invoke the fact that the bistable range of the subcritical Hopf bifurcation shrinks to zero in the singular limit (see for instance \cite[Theorem 3.1]{KRSZ01a}). Type III excitable models are, by definition, always globally asymptotically stable.

\subsection{Singularly perturbed spike latency in Type IV and V excitable systems}

Spike latency appears when the trajectory travels the ``ghost'' of the center manifold of the saddle node bifurcation. It is prominent in Type IV and V excitable models, since in these models this center manifold is attractive. Indeed, the recovery variable nullcline being strictly monotone increasing, the saddle-node bifurcation lies on the lower attractive branch of the voltage nullcline (see Figure \ref{FIG: latency}). By standard persistence arguments, this implies that its center manifold $W_c^\epsilon$ is strongly attractive near the singular limit. A consequence of this attractiveness is that, after the bifurcation at $I_{app}=I_{SN}$, an originally resting trajectory is attracted toward the ghost of $W_c^\epsilon$ between the two nullclines, where the vector field magnitude is proportional to $I_{app}-I_{SN}$. The passage time thus diverges to infinity as $I_{app}\searrow I_{SN}$, corresponding to a prominent and robust spike-latency.
\begin{figure}
\centering
\includegraphics[width=\textwidth]{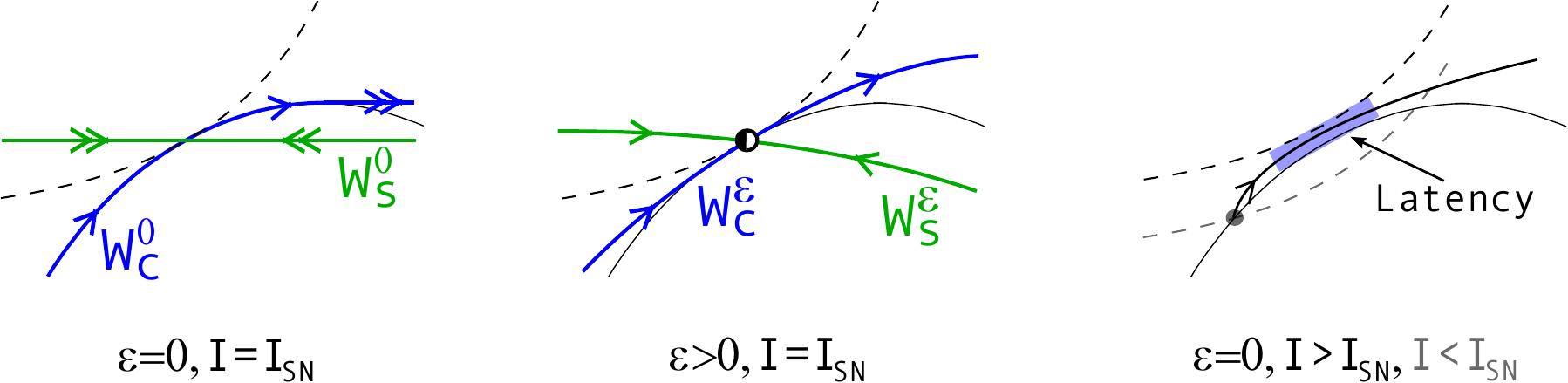}
\caption{Mechanism of spike latency in Types IV and V excitable models. Left: slow-fast dynamics near the saddle-node bifurcation. Center: stable $W_s^\epsilon$ and center $W_c^\epsilon$ manifolds of the bifurcation for $\epsilon>0$. Right: the ghost of $W_c^\epsilon$ (depicted as the blue rectangle) is attractive for $I>I_{SN}$ corresponding to prominent and robust spike latency.}\label{FIG: latency}
\end{figure}

\subsubsection*{Absence of singularly perturbed spike-latency in competitive models}

As opposed to Types IV and V, the saddle node bifurcation of Type I excitable models lies on the upper repulsive (right) branch of the voltage nullcline (recall again that the recovery variable nullcline is strictly monotone increasing). Hence, attractiveness of the bifurcation center manifold, and thus of its ghost, is lost near the singular limit. Even though for finite $\epsilon$ this center manifold can be attractive (see Figure \ref{FIG: no latency}), its ghost (between the two nullclines) does not attract an originally resting trajectory: the trajectory necessarily travels below the recovery variable nullcline  where the vector field has finite magnitude independently of the distance from the bifurcation. These simple arguments show that competitive excitable models and, in particular, Type I excitable models can not exhibit spike latency.
\begin{figure}
\centering
\includegraphics[width=\textwidth]{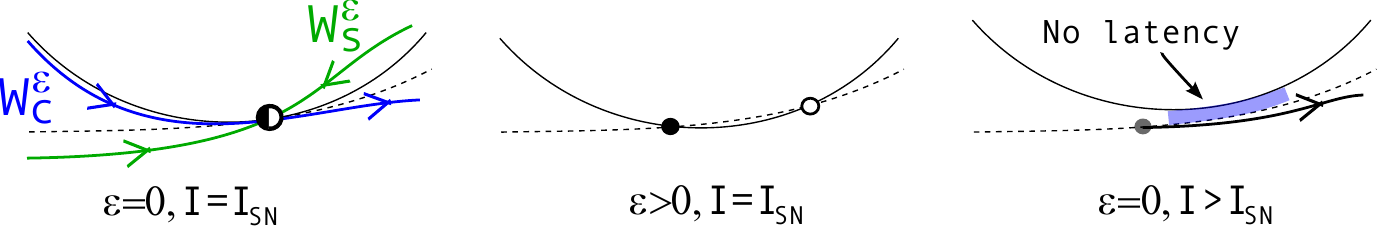}
\caption{Absence of spike latency in Type I excitable models. Left: the saddle-node bifurcation in Type I excitable models with its center $W_c^\epsilon$ and stable $W_s^\epsilon$ manifolds. Center: the stable and the saddle fixed points before the bifurcation. Right: the ghost of $W_c^\epsilon$ (blue rectangle) after the biurcation and the behavior of an originally resting trajectory.}\label{FIG: no latency}
\end{figure}

\section{Discussion}
\label{SEC: discussion}
\subsection*{Mirroring the FitzHugh-Nagumo equation accounts for cooperative gating variables}

Mirroring the FitzHugh-Nagumo equation was motivated by the inclusion, in a simple model of neuronal excitability, of the transcritical bifurcation recently observed in a planar reduction of the Hodgkin-Huxley model augmented with an activating calcium current. This purely geometrical construction is actually tightly linked with the underlying electrophysiology: in the upper FitzHugh-Nagumo like part of the phase portrait the model is competitive, as it is the original Hodgkin-Huxley model; in the mirrored part it is cooperative, accounting for cooperative gating variables, such as the activation of calcium current. To the best of our knowledge, this unified picture is not present in existing planar models of neuronal excitability, that are purely competitive.

\subsection*{Cooperativity unravels a pitchfork bifurcation organizing old and new types of excitability}

The distinctive effects of cooperative gating variables on neuronal excitability are widely studied in high dimensional conductance-based models and in {\it in vitro} recordings. However, these same signatures can not be reproduced in competitive models and the underlying dynamical mechanisms have remained obscure to date.

The unfolding of a pitchfork bifurcation organizing the proposed model reveals how the inclusion of cooperative variables changes neuronal excitability. The obtained parameter chart recovers the three known types of (competitive) excitability and unmasks two novel types that we naturally defined as Types IV and V. The defining condition of Types IV and V excitability is cooperativity, {\it i.e.}
\begin{equation*}
\frac{\partial\dot V}{\partial n}\frac{\partial\dot n}{\partial V}>0,
\end{equation*}
at the hyperpolarized stable steady state. This sole condition ensures the presence of the electrophysiological signatures of cooperative gating variables, in particular, bistability. In a forthcoming publication, we will show that this defining condition also holds in higher-dimensional models and is not an artifact of planar reductions.
 
\subsection*{Bistability and bursting in singularly perturbed excitable models}
\label{SEC: bist and burst in SP systems}

A distinct feature of Types IV and V excitable models with respect to the three other types is the existence of a finite range of bistability.
%\footnote{see Theorem \ref{THM: KRSZ01 thm reformulation} and the discussion in Section \ref{SEC: SP bistability}} for all values of the singular perturbation parameter $0<\epsilon\ll1$.
%More precisely, Theorem 1 and the discussion in Section \ref{SEC: SP bistability} ensure that the bistability range in both Types IV and V is independent of $\epsilon$, modulo variations that are $\mathcal O(\sqrt{\epsilon})$.
Bistability has been described in the context of Types I and II excitability as well (see e.g. \cite{Rinzel:1989:ANE:94605.94613}), but it was shown in Section \ref{SEC: SP bistability} that, in all these situations, the bistability range shrinks to zero as the timescale separation is increased. In contrast, the stable manifold of the saddle point that separates the two basin of attraction in Types IV and V excitability is a robust (hyperbolic) geometric object that persists in the singular limit. Because neurons do exhibit a pronounced timescale separation, the robustness of the bistable range in Types IV and V is thought to be an important feature of excitable models that are not purely competitive.

The relevance of bistability in excitable models lies in its relevance for model bursting. Bursting is typically the result of a slow adaptation variable that modulates the applied current across the bistability range, creating a hysteresis loop between the stable down-state and the up-state attractor. An important conjecture derived from our analysis is that bursting will persist near the singularly perturbed limit of model (\ref{EQ: mirrored FHN dynamics}) only in Type IV and V excitability, that is only in the presence of cooperative ionic channels. In other words, Types IV and V excitability would be the essential sources of bursting in singularly perturbed models.

\section*{Methods}

The parameter chart in Figures \ref{FIG: TC PC} and \ref{FIG: TC exc types} have been numerically drawn using MATLAB\footnote{{\tt http://www.mathworks.com}} and modified with the Open Source vector graphics editor Inkscape\footnote{\tt http://inkscape.org}.\\
Phase portrait were hand-drawn with Inkscape.\\
The bifurcation diagrams in Figures \ref{FIG: Type IV}, \ref{FIG: Type Va}, and \ref{FIG: Type Vb} have been obtained with XPP environment\footnote{\tt http://www.math.pitt.edu/$\thicksim$bard/xpp/xpp.html}.

\bibliographystyle{unsrt}
\bibliography{refs}

\newcommand{\SortNoop}[1]{} % \def\loria{Lor\'{\i}a\,} %
  \def\nesic{Ne\v{s}i\'{c}\,} %
  \def\astrom{{\SortNoop{As}\AA}str{\"{o}}m\,}\let\c=\cedille
\begin{thebibliography}{10}

\bibitem{hodgkin1948local}
A.L. Hodgkin.
\newblock The local electric changes associated with repetitive action in a
  non-medullated axon.
\newblock {\em The Journal of physiology}, 107(2):165--181, 1948.

\bibitem{fitzhugh61}
R.~FitzHugh.
\newblock Impulses and physiological states in theoretical models of nerve
  membrane.
\newblock {\em Biophysical J.}, 1:445--466, 1961.

\bibitem{rinzel1985excitation}
J.~Rinzel.
\newblock Excitation dynamics: insights from simplified membrane models.
\newblock In {\em Fed. Proc}, volume~44, pages 2944--2946, 1985.

\bibitem{Rinzel:1989:ANE:94605.94613}
J.~Rinzel and G.~B. Ermentrout.
\newblock {\em Analysis of neural excitability and oscillations}, pages
  135--169.
\newblock MIT Press, Cambridge, MA, USA, 1989.

\bibitem{DRFRSESE_PLoS_sub}
G.~Drion, A.~Franci, V.~Seutin, and R.~Sepulchre.
\newblock A novel phase portrait for neuronal excitability.
\newblock 2012.
\newblock Submitted to: PLoS 1. Preprint available at: {\tt
  http://arxiv.org/abs/1112.2588}.

\bibitem{SEYDEL94}
R.~Seydel.
\newblock {\em Practical bifurcation and stability analysis}, volume~5 of {\em
  Interdisciplinary Applied Mathematics}.
\newblock Springer-Verlag, New York, third edition, 2010.

\bibitem{smith2008monotone}
H.L. Smith.
\newblock {\em Monotone dynamical systems: An introduction to the theory of
  competitive and cooperative systems}.
\newblock American Mathematical Soc., 2008.

\bibitem{plant1976mathematical}
R.E. Plant and M.~Kim.
\newblock Mathematical description of a bursting pacemaker neuron by a
  modification of the hodgkin-huxley equations.
\newblock {\em Biophysical journal}, 16(3):227--244, 1976.

\bibitem{McCormick92a}
D~A McCormick and J~R Huguenard.
\newblock A model of the electrophysiological properties of thalamocortical
  relay neurons.
\newblock {\em J Neurophysiol}, 68(4):1384--1400, 1992.

\bibitem{destexhe1996vivo}
A.~Destexhe, D.~Contreras, M.~Steriade, T.J. Sejnowski, and J.R. Huguenard.
\newblock In vivo, in vitro, and computational analysis of dendritic calcium
  currents in thalamic reticular neurons.
\newblock {\em The Journal of neuroscience}, 16(1):169--185, 1996.

\bibitem{Traub01081991}
R.~D. Traub, R.~K. Wong, R.~Miles, and H.~Michelson.
\newblock A model of a ca3 hippocampal pyramidal neuron incorporating
  voltage-clamp data on intrinsic conductances.
\newblock {\em Journal of Neurophysiology}, 66(2):635--650, 1991.

\bibitem{ERTE10}
G.~B. Ermentrout and D.~H. Terman.
\newblock {\em Mathematical Foundations of Neuroscience}.
\newblock Interdisciplinary Applied Mathematics. Springer, 2010.

\bibitem{IZHIKEVICH2007}
E.~M. Izhikevich.
\newblock {\em Dynamical Systems in Neuroscience: The Geometry of Excitability
  and Bursting}.
\newblock MIT Press, 2007.

\bibitem{zhan1999current}
X.J. Zhan, C.L. Cox, J.~Rinzel, and S.M. Sherman.
\newblock Current clamp and modeling studies of low-threshold calcium spikes in
  cells of the cat’s lateral geniculate nucleus.
\newblock {\em Journal of neurophysiology}, 81(5):2360--2373, 1999.

\bibitem{tateno2004threshold}
T.~Tateno, A.~Harsch, and H.P.C. Robinson.
\newblock Threshold firing frequency--current relationships of neurons in rat
  somatosensory cortex: type 1 and type 2 dynamics.
\newblock {\em Journal of neurophysiology}, 92(4):2283--2294, 2004.

\bibitem{sessley2002evidence}
S.~Sessley and R.J. Butera.
\newblock Evidence for type i excitability in molluscan neurons.
\newblock In {\em Engineering in Medicine and Biology, 2002. 24th Annual
  Conference and the Annual Fall Meeting of the Biomedical Engineering Society,
  EMBS/BMES Conference, 2002. Proceedings of the Second Joint}, volume~3, pages
  1966--1967. IEEE, 2002.

\bibitem{wechselberger2007canards}
M.~Wechselberger.
\newblock Canards.
\newblock {\em Scholarpedia}, 2(4):1356, 2007.

\bibitem{Messina197657}
C.~Messina and R.~Cotrufo.
\newblock Different excitability of type 1 and type 2 alpha-motoneurones: The
  recruitment curve of h- and m-responses in slow and fast muscles of rabbits.
\newblock {\em Journal of the Neurological Sciences}, 28(1):57 -- 63, 1976.

\bibitem{Gai01122009}
Yan Gai, Brent Doiron, Vibhakar Kotak, and John Rinzel.
\newblock Noise-gated encoding of slow inputs by auditory brain stem neurons
  with a low-threshold k+ current.
\newblock {\em Journal of Neurophysiology}, 102(6):3447--3460, 2009.

\bibitem{clay2008simple}
J.R. Clay, D.~Paydarfar, and D.B. Forger.
\newblock A simple modification of the hodgkin and huxley equations explains
  type 3 excitability in squid giant axons.
\newblock {\em Journal of The Royal Society Interface}, 5(29):1421--1428, 2008.

\bibitem{Hallworth20082003}
Nicholas~E. Hallworth, Charles~J. Wilson, and Mark~D. Bevan.
\newblock Apamin-sensitive small conductance calcium-activated potassium
  channels, through their selective coupling to voltage-gated calcium channels,
  are critical determinants of the precision, pace, and pattern of action
  potential generation in rat subthalamic nucleus neurons in vitro.
\newblock {\em The Journal of Neuroscience}, 23(20):7525--7542, 2003.

\bibitem{huguenard1992novel}
J.R. Huguenard and D.A. Prince.
\newblock A novel {T-type} current underlies prolonged {Ca(2+)}-dependent burst
  firing in {GABAergic} neurons of rat thalamic reticular nucleus.
\newblock {\em The Journal of neuroscience}, 12(10):3804--3817, 1992.

\bibitem{mccormick1997sleep}
D.A. McCormick and T.~Bal.
\newblock Sleep and arousal: thalamocortical mechanisms.
\newblock {\em Annual review of neuroscience}, 20(1):185--215, 1997.

\bibitem{johnson1992burst}
S.W. Johnson, V.~Seutin, and R.A. North.
\newblock Burst firing in dopamine neurons induced by n-methyl-d-aspartate:
  role of electrogenic sodium pump.
\newblock {\em Science}, 258(5082):665--667, 1992.

\bibitem{grace1984control}
A.A. Grace and B.S. Bunney.
\newblock The control of firing pattern in nigral dopamine neurons: burst
  firing.
\newblock {\em The Journal of neuroscience}, 4(11):2877--2890, 1984.

\bibitem{gray1996chattering}
C.M. Gray and D.A. McCormick.
\newblock Chattering cells: superficial pyramidal neurons contributing to the
  generation of synchronous oscillations in the visual cortex.
\newblock {\em Science}, 274(5284):109--113, 1996.

\bibitem{heyward2001membrane}
P.~Heyward, M.~Ennis, A.~Keller, and M.T. Shipley.
\newblock Membrane bistability in olfactory bulb mitral cells.
\newblock {\em The Journal of Neuroscience}, 21(14):5311--5320, 2001.

\bibitem{wilson1981spontaneous}
C.J. Wilson and P.M. Groves.
\newblock Spontaneous firing patterns of identified spiny neurons in the rat
  neostriatum.
\newblock {\em Brain research}, 220(1):67--80, 1981.

\bibitem{wilson2008up}
C.~Wilson.
\newblock Up and down states.
\newblock {\em Scholarpedia journal}, 3(6):1410, 2008.

\bibitem{JONES95}
C.~K.~R. Jones.
\newblock Geometric singular perturbation theory.
\newblock In {\em Dynamical systems. Springer Lecture Notes in Math. 1609},
  pages 44--120, Berlin, 1995. Springer.

\bibitem{KRSZ2001relax}
M.~Krupa and P.~Szmolyan.
\newblock Relaxation oscillation and canard explosion.
\newblock {\em J. Differential Equations}, 174(2):312--368, 2001.

\bibitem{KRSZ01}
M.~Krupa and P.~Szmolyan.
\newblock Extending slow manifolds near transcritical and pitchfork
  singularities.
\newblock {\em Nonlinearity}, 14:1473--1491, 2001.

\bibitem{KRSZ01a}
M.~Krupa and P.~Szmolyan.
\newblock Extending geometrical singular perturbation theory to nonhyperbolic
  points - folds and canards points in two dimensions.
\newblock {\em SIAM J. Math. Analysis}, 33(2):286--314, 2001.

\bibitem{FENICHEL79}
N.~Fenichel.
\newblock Geometric singular perturbation theory.
\newblock {\em J. Diff. Eq.}, 31:53--98, 1979.

\bibitem{morris1981voltage}
C.~Morris and H.~Lecar.
\newblock Voltage oscillations in the barnacle giant muscle fiber.
\newblock {\em Biophysical journal}, 35(1):193--213, 1981.

\end{thebibliography}

\end{document}